%% file: main.tex
\documentclass{article}
\usepackage[utf8]{inputenc}
\usepackage[english]{babel}
\usepackage{authblk}

\title{A Catalog of Formulations for the Network Pricing Problem}
\author[1]{Quang Minh Bui}
\author[1]{Bernard Gendron}
\author[1]{Margarida Carvalho}
\affil[1]{CIRRELT and D\'epartement d'informatique et de recherche op\'erationnelle, Universit\'e de Montr\'eal}
\date{}

\usepackage{mathtools}
\usepackage{amsmath}
\usepackage{amssymb}
\usepackage{amsthm}
\usepackage{xfrac}
\usepackage{multirow}
\usepackage{enumitem}
\usepackage{diagbox}
\usepackage{makecell}
\usepackage{subcaption}
\usepackage{algorithm}
\usepackage{algorithmic}
\usepackage{booktabs}
\usepackage{bm}
\usepackage{hyperref}

\usepackage{tikz}
\usepackage{pgfplots}
\usetikzlibrary{patterns}
\usetikzlibrary{arrows.meta}
\usetikzlibrary{positioning}
\pgfplotsset{compat=1.15}

\newtheorem{theorem}{Theorem}
\newtheorem{lemma}[theorem]{Lemma}
\newtheorem{corollary}[theorem]{Corollary}
\newtheorem{definition}[theorem]{Definition}
\newtheorem{assumption}[theorem]{Assumption}
\newtheorem{example}[theorem]{Example}

\newcommand{\tabitem}{~~\llap{\textbullet}~~}

\newcommand{\st}{\text{s.t. }}

\newcommand{\sumk}{\sum_{k \in K}}
\newcommand{\suma}[1]{\sum_{a \in A_{#1}}}

\newcommand{\sumai}[2]{\sum_{a \in A^{#1}_{#2}(i)}}

\newcommand{\x}{x_a^k}
\newcommand{\y}{y_a^k}
\newcommand{\bx}{\hat{x}_a}
\newcommand{\by}{\hat{y}_a}
\newcommand{\ca}{c_a}
\newcommand{\Ta}{T_a}
\newcommand{\Tak}{t_a^k}

\newcommand{\lmk}{\lambda^k}

\newcommand{\dap}{\delta_a^p}
\newcommand{\daq}{\delta_a^q}

\newcommand{\ok}{{o^k}}
\newcommand{\dk}{{d^k}}

\newcommand{\Pk}{P^k}

\newcommand{\pk}{p}

\newcommand{\pik}{\pi^k}
\newcommand{\sump}{\sum_{\pk \in \Pk}}

\newcommand{\subad}[1]{\substack{a \in A_{#1} \\ \dap = 1}}

\newcommand{\zpk}{z^k_{\pk}}

\newcommand{\tauk}{\tau^k}
\newcommand{\qi}{q^i}
\newcommand{\hqi}{\hat{q}^i}

\newcommand{\ps}[1]{p^{(#1)}}
\newcommand{\psj}{\ps{j}}

\newcommand{\ha}{\hat{a}}

\newcommand{\prog}[2][]{(\text{#2}^{#1})}

\newcommand{\ie}{\emph{i.e.}}

\begin{document}

\maketitle

\begin{abstract}
    We study the network pricing problem where the leader maximizes their revenue by
    determining the optimal amounts of tolls to charge on a set of arcs,
    under the assumption that the followers will react rationally and choose the shortest paths to travel.
    Many distinct single-level reformulations to this bilevel optimization program have been proposed,
    however, their relationship has not been established.
    In this paper, we aim to build a connection between those reformulations and
    explore the combination of the path representation with various modeling options,
    allowing us to generate 12 different reformulations of the problem.
    Moreover, we propose a new path enumeration scheme, path-based preprocessing,
    and hybrid framework to further improve performance and robustness when solving the final model.
    We provide numerical results, comparing all the derived reformulations
    and confirming the efficiency of the novel dimensionality reduction procedures.
\end{abstract}

\section{Introduction}
The network pricing problem (NPP) is a bilevel optimization program involving two parties:
a leader and multiple followers. First, the leader sets the prices of several arcs (called tolled arcs) in a network.
Afterwards, the followers choose the optimal paths in the network subject to the prices set by the leader.
The leader's objective is to maximize their profit, while the goal of the followers
is to minimize their own costs of transit across the network. High prices do not always
mean more profit for the leader if only a few followers can afford them. Thus, the leader must
seek a balance between the prices and the demands of the followers reacting to those prices.
In practice, the leader could be a highway authority and the followers could be groups of the population
residing in different neighborhoods or cities.

\subsection{Related Literature}

The NPP was first introduced in Labbé et al. \cite{labbe1998} as a single-commodity problem (one follower),
and later extended to the multi-commodity version (multiple followers) by Brotcorne et al. \cite{brotcorne2001}.
The problem is proven to be NP-hard even when there is only one follower \cite{roch2005}. The first mixed integer linear program (MILP) formulation
was also introduced in \cite{brotcorne2001}. Since then, many improvements to the MILP of \cite{brotcorne2001}
were proposed such as the shortest-path graph model (SPGM) \cite{vanhoesel2003}, preprocessing techniques \cite{bouhtou2007},
path-based models \cite{bouhtou2007, didibiha2006}, valid inequalities,
and tight bounds for big-M parameters \cite{dewez2008}. Other methods were also explored, including multipath enumeration
\cite{brotcorne2011} and tabu search \cite{brotcorne2012}. Variants of the network pricing problem include
the joint network design and pricing problem \cite{brotcorne2008}, the complete toll NPP (clique pricing problem)
\cite{heilporn2010}, and the logit NPP \cite{gilbert2015}.

The NPP has a natural bilevel optimization formulation.
There are general-purpose methods to solve bilevel optimization problems \cite{fischetti2017, tahernejad2020}.
However, due to their complexity, they are usually considered per case to exploit specific problem structures.
A common approach to solve a bilevel optimization problem is to convert it to a single-level reformulation.
For the NPP, the follower problems are linear. Thus, there are three
methods generally used for this conversion based on optimality conditions: strong duality, complementary slackness, and value function.
In the first two methods, the primal and the dual constraints of the follower problems are included in the single-level reformulation
and then accompanied by either strong duality or complementary slackness constraints.
In the third method, only the primal constraints are needed and the optimality of the follower problems
is ensured by the value function constraints.

Strong duality was used in the original formulation proposed in \cite{brotcorne2001}.
The discovery of the path representation led to the use of complementary slackness in the path-based formulations developed in \cite{didibiha2006}.
Value function formulations were mentioned in \cite{heilporn2006} and \cite{bouhtou2007}.
However, these models were presented as separate formulations with little connection in between.

\subsection{Contribution and Paper Organization}

In the first part of this work, our aim is to explore the links between all of the previously mentioned single-level reformulations of the NPP.
First, we show that in strong duality and complementary slackness reformulations, the
path representation can be applied not only to the primal follower problems, but also to the dual problems.
This allows us to show that the value function method is a special case of strong duality when
the dual problems are written in the path representation. Second, we provide a systematic way to formalize
a reformulation by breaking it down into components with different options for each component. This general method
can also be used to categorize all the models mentioned above and explain the connection between them.

The second part of this work is dedicated to the path enumeration process. This is important because
a well-performed path-based reformulation requires an efficient path enumeration algorithm.
Existing methods must enumerate either all the paths \cite{bouhtou2007}
or a set of relevant paths but require multiple calls to a linear solver \cite{didibiha2006}.
Based on previous discoveries on the properties of redundant paths in \cite{bouhtou2007, didibiha2006},
we develop a new path enumeration process that can enumerate relevant paths without employing a linear solver.
Besides that, based on this set of relevant paths, we also derive a new preprocessing method that can be applied to arc-based reformulations.
This preprocessing method is crucial if for a fair comparison between the reformulations, since
path enumeration is also counted as a preprocessing method for the path-based reformulations.

Regardless of the efficiency of the path enumeration process, the potential size of the set of all relevant paths
is exponential, which may require an enormous amount of time to enumerate them.
In the last part of this work, we introduce a method to circumvent this problem by taking advantage of the
multi-commodity nature of the NPP. By deciding which commodities are worth applying path enumeration
and mixing different reformulations into one, we can save time by spending less time enumerating very large sets of paths
and more time solving the actual optimization model. In this paper, we refer to this solution as the hybrid framework for multi-commodity problems.

This paper is structured as follows. In Section 2,
we describe the NPP, namely, its bilevel optimization program, and explain the general method to formalize a reformulation.
Section 3 shows how path enumeration is performed in the novel preprocessing method.
The hybrid framework is introduced in Section 4.
Section 5 presents computational results, including a comparison of all formulations
and several experiments to prove the efficiency of the new preprocessing method and the hybrid framework.

\section{The Network Pricing Problem}
In this section, we provide the bilevel programming formulation for the network pricing problem in Section 2.1
and a systematic way to generate single-level reformulations through the combination of different components in Section 2.2.
The conversion from the bilevel formulation to single-level reformulations produces bilinear terms.
The linearization of these bilinear terms is the topic of discussion in Section 2.3.

\subsection{Problem Formulation}

Let us consider a graph \(G = (V, A)\) where \(A\) is partitioned into a set of toll arcs \(A_1\)
and a set of toll-free arcs \(A_2\). Each arc \(a \in A\) has an initial cost \(c_a > 0\).
Let \(K\) be the set of commodities (O-D pairs) and \(o^k, d^k, \eta^k\) be the origin, the destination,
and the demand of commodity \(k \in K\), respectively. We define the following variables:
\begin{itemize}
    \item \(\Ta, a \in A_1\): the toll of arc \(a\);
    \item \(\x, a \in A_1, k \in K\): the flow of commodity \(k\) on toll arc \(a\);
    \item \(\y, a \in A_2, k \in K\): the flow of commodity \(k\) on toll-free arc \(a\).
\end{itemize}

The network pricing problem is then formulated as a bilevel program:
\begin{align*}
    \prog{NPP}\quad \max_{T \geq 0,x,y} & \sumk \suma1 \eta^k \Ta\x                 \\
    \st \forall k \in K             & \left\{ \begin{aligned}
        (x^k, y^k) \in \arg\min_{\hat{x}, \hat{y}} & \suma1 (\ca+\Ta)\bx + \suma2 \ca\by,                                                      \\
        \st                                        & \sumai{+}1 \bx + \sumai{+}2 \by - \sumai{-}1 \bx - \sumai{-}2 \by = b_i^k &  & i \in V,   \\
                                                   & \bx \in \{0, 1\}                                                          &  & a \in A_1, \\
                                                   & \by \in \{0, 1\}                                                          &  & a \in A_2,
    \end{aligned} \right.
\end{align*}
where \(b_i^k = 1\) if \(i = o^k\), \(-1\) if \(i = d^k\), and \(0\) otherwise. For a node $i$, the set of its outgoing  tolled (toll-free) arcs is \(A^+_1(i)\) (\(A^+_2(i)\))
and the set of its incoming tolled (toll-free) arcs is \(A^-_1(i)\) (\(A^-_2(i)\)).
The leader controls the toll prices \(\Ta\), while each follower \(k \in K\) decides the taken paths by setting \(\x\) and \(\y\) equal to 1
if the arc belongs to the path, and 0 otherwise.
We consider the optimistic version of this bilevel problem, \ie, if a follower has multiple optimal solutions
with respect to the prices set by the leader, then they will choose the solution that benefits the leader the most.
For each commodity, we assume that there exists at least a toll-free path (a path without tolled arcs).
Otherwise, the leader could gain infinite revenue by exploiting that particular O-D pair.

\subsection{Generalized Single-level Reformulations}

A generalized single-level reformulation of the network pricing problem consists of three components:
\begin{itemize}
    \item Primal representation (arc or path);
    \item Dual representation (arc or path);
    \item Optimality condition (strong duality or complementary slackness);
\end{itemize}

First, the primal and the dual representations of the follower problems need to be chosen.
There are two options for each: the arc representation and the path representation.
Consider the follower problem for commodity \(k\):
\begin{align}
    \prog[k]{PA} &  & \min_{x^k, y^k} & \suma1 (\ca+\Ta)\x + \suma2 \ca\y \nonumber                                                                     \\
                 &  & \st             & \sumai{+}1 \x + \sumai{+}2 \y - \sumai{-}1 \x - \sumai{-}2 \y = b_i^k &  & i \in V,                  \label{pa} \\
                 &  &                 & \x \in \{0, 1\}                                                       &  & a \in A_1, \nonumber                 \\
                 &  &                 & \y \in \{0, 1\}                                                       &  & a \in A_2. \nonumber
\end{align}

This is called the \emph{primal-arc representation} of the follower problem.
Because the set of constraints of the follower problem forms a totally unimodular matrix,
integrality conditions for \(\x\) and \(\y\) can be dropped. This enables us to write
the follower problem in the \emph{dual-arc representation}:
\begin{align}
    \prog[k]{DA} &  & \max_{\lmk} \  & \lmk_\ok - \lmk_\dk \nonumber                                             \\
                 &  & \st            & \lmk_i - \lmk_j \leq \ca + \Ta &  & a \equiv (i, j) \in A_1,  \label{da1} \\
                 &  &                & \lmk_i - \lmk_j \leq \ca       &  & a \equiv (i, j) \in A_2.  \label{da2}
\end{align}

In \cite{bouhtou2007, didibiha2006}, instead of writing the follower problem using arc-flow \(\x, \y\), the authors
replaced these variables with \(\zpk\) representing the path-flow or the selection of paths \(p\) in the set of all elementary paths \(\Pk\)
which is finite. We will refer this as the \emph{primal-path representation}:
\begin{align}
    \prog[k]{PP} &  & \min_{z^k} \  & \sump \left( \suma{} \dap\ca + \suma1 \dap\Ta \right)\zpk \nonumber                           \\
                 &  & \st           & \sump \zpk = 1,                                                     &  & \label{pp}           \\
                 &  &               & \zpk \in \{0, 1\}                                                   &  & p \in \Pk. \nonumber
\end{align}

In the above formulation, \(\dap = 1\) if the path \(p\) includes the arc \(a\) and \(\dap = 0\) if not.
Once again, the constraints form a totally unimodular matrix, hence \(\zpk\) is not required to be binary and
thus the follower problem also has a \emph{dual-path representation}:
\begin{align}
    \prog[k]{DP} &  & \max_{L^k} \  & L^k \nonumber                                                                        \\
                 &  & \st           & L^k \leq  \suma{} \dap\ca + \suma1 \dap\Ta &  & p \in \Pk.                \label{dp}
\end{align}

In total, there are 4 combinations of primal-dual representations (arc-arc, arc-path, path-arc, and path-path).
Note that the primal and the dual representations are not required to match each other,
so primal-arc could also be paired with dual-path.

After deciding the representations, we need to connect the primal and the dual representations
by either strong duality or complementary slackness constraints.
Strong duality simply connects the primal and the dual objective functions together.
Below are the strong duality constraints for the arc-arc, arc-path, path-arc, and path-path combinations, respectively:
\begin{align}
    \suma1 (\ca+\Ta)\x + \suma2 \ca\y                         & = \lmk_{o^k} - \lmk_{d^k}, \label{sd-aa} \\
    \suma1 (\ca+\Ta)\x + \suma2 \ca\y                         & = L^k,                     \label{sd-ap} \\
    \sump \left( \suma{} \dap\ca + \suma1 \dap\Ta \right)\zpk & = \lmk_{o^k} - \lmk_{d^k}, \label{sd-pa} \\
    \sump \left( \suma{} \dap\ca + \suma1 \dap\Ta \right)\zpk & = L^k.                     \label{sd-pp}
\end{align}

Complementary slackness matches bounded variables in the primal with inequality constraints in the dual.
The pairing is simple if the primal and the dual have the same representation (either arc-arc or path-path):
\begin{align}
    (\ca + \Ta - \lmk_i + \lmk_j)\x                           & = 0 &  & a \equiv (i, j) \in A_1, \label{cs-aa1} \\
    (\ca - \lmk_i + \lmk_j)\y                                 & = 0 &  & a \equiv (i, j) \in A_2, \label{cs-aa2} \\
    \left(\suma{} \dap\ca + \suma1 \dap\Ta - L^k \right) \zpk & = 0 &  & p \in \Pk.          \label{cs-pp}
\end{align}

However, if the primal and the dual have different representations, variables in the primal representation must be converted
to their dual counterparts. For the path-arc combination, conversion from \(z\) to \((x, y)\) is done by using the identities
\(\x, \y = \sump \dap\zpk\):
\begin{align}
    (\ca + \Ta - \lmk_i + \lmk_j)\left(\sump \dap\zpk \right) & = 0 &  & a \equiv (i, j) \in A_1, \label{cs-pa1} \\
    (\ca - \lmk_i + \lmk_j)\left(\sump \dap\zpk \right)       & = 0 &  & a \equiv (i, j) \in A_2. \label{cs-pa2}
\end{align}

For the arc-path combination, the conversion from \((x, y)\) to \(z\) is more complicated. We will use the definition that
\(\zpk = 1\) if and only if \(\x = 1\) and \(\y = 1\) for all arcs \(a\) belonging to the path \(p\) (given that \((x, y)\) produce an elementary path).
Mathematically, \(\zpk = \prod_{a \in A_1 \mid \dap = 1} \x \prod_{a \in A_2 \mid \dap = 1} \y \).
The complementary slackness constraint for the arc-path combination is:
\begin{align}
    \left(\suma{} \dap\ca + \suma1 \dap\Ta - L^k \right) \prod_{\subad1} \x \prod_{\subad2} \y & = 0 &  & p \in \Pk. \label{cs-ap}
\end{align}

An important rule of the conversion is that it must cover all possible solutions in the primal representation.
Otherwise, the optimality condition can be dodged by choosing the primal solution that is not covered.
If we leave the graph and \(\Pk\) as inputted, this rule is satisfied. However, in Section 3, we will introduce methods that
reduce the size of the graph and \(\Pk\), thus requiring additional attention. We will discuss this later in Section 3.2.

For completeness, accordingly with the choice of the primal representation, here are the objective function of the single-level formulations:
\begin{align}
     & \sumk \suma1 \eta^k \Ta\x,                                  \label{obj-x} \\
     & \sumk \suma1 \sump \eta^k \dap\Ta\zpk. \label{obj-z}
\end{align}

Based on the primal, the dual, and the optimality condition, we can categorize all the single-level formulations as in Table \ref{tab:map}.
The variables, constraints and objective functions of these formulations are listed in Table \ref{tab:content}.
The variables are bounded implicitly (\(x, y, z, T\) are non-negative; \(\lambda, L\) are unbounded).

\begin{table}[h]
    \begin{subtable}[h]{\textwidth}
        \centering
        \begin{tabular}{|c||c|c|}
            \hline
            \diagbox{Dual}{Primal} & Arc                 & Path                          \\ \hline \hline
            Arc                    & Standard (STD)      & \makecell{Path-Arc Standard   \\ (PASTD)}       \\ \hline
            Path                   & Value Function (VF) & \makecell{Path Value Function \\ (PVF)} \\ \hline
        \end{tabular}
        \subcaption{Strong duality}
    \end{subtable}
    \par\bigskip
    \begin{subtable}[h]{\textwidth}
        \centering
        \begin{tabular}{|c||c|c|}
            \hline
            \diagbox{Dual}{Primal} & Arc                      & Path \\ \hline \hline
            Arc                    & \makecell{Complementary         \\ Slackness (CS)}  & \makecell{Path-Arc \\ Complementary \\ Slackness (PACS)}       \\ \hline
            Path                   & \makecell{Value Function        \\ Complementary        \\ Slackness (VFCS)} & \makecell{Path Complementary \\ Slackness (PCS)} \\ \hline
        \end{tabular}
        \subcaption{Complementary slackness}
    \end{subtable}
    \caption{Complete map of all single-level reformulations.}
    \label{tab:map}
\end{table}

\begin{table}[h]
    \centering
    \begin{tabular}{llllll}
        \toprule
              &                       &               & \multicolumn{3}{c}{Constraints}                                        \\
        \cmidrule(lr){4-6}
        Label & Variables             & Objective     & Primal     & Dual                     & Opt. Cond.                     \\
        \midrule
        \textbf{Strong duality}\span\span\span                                                                                 \\
        STD   & \(T, x, y, \lambda \) & \eqref{obj-x} & \eqref{pa} & \eqref{da1}, \eqref{da2} & \eqref{sd-aa}                  \\
        VF    & \(T, x, y, L \)       & \eqref{obj-x} & \eqref{pa} & \eqref{dp}               & \eqref{sd-ap}                  \\
        PASTD & \(T, z, \lambda \)    & \eqref{obj-z} & \eqref{pp} & \eqref{da1}, \eqref{da2} & \eqref{sd-pa}                  \\
        PVF   & \(T, z, L \)          & \eqref{obj-z} & \eqref{pp} & \eqref{dp}               & \eqref{sd-pp}                  \\
        \midrule
        \textbf{Complementary slackness}\span\span\span                                                                        \\
        CS    & \(T, x, y, \lambda \) & \eqref{obj-x} & \eqref{pa} & \eqref{da1}, \eqref{da2} & \eqref{cs-aa1}, \eqref{cs-aa2} \\
        VFCS  & \(T, x, y, L \)       & \eqref{obj-x} & \eqref{pa} & \eqref{dp}               & \eqref{cs-ap}                  \\
        PACS  & \(T, z, \lambda \)    & \eqref{obj-z} & \eqref{pp} & \eqref{da1}, \eqref{da2} & \eqref{cs-pa1}, \eqref{cs-pa2} \\
        PCS   & \(T, z, L \)          & \eqref{obj-z} & \eqref{pp} & \eqref{dp}               & \eqref{cs-pp}                  \\
        \bottomrule
    \end{tabular}
    \caption{List of variables, constraints, and objective functions of all single-level reformulations.}
    \label{tab:content}
\end{table}

There are some notable reformulations in the list.
The most canonical way to convert a bilevel linear program to its
single-level version is the standard formulation (STD):
\begin{align*}
    \prog{STD} &  & \max\  & \sumk \suma1 \eta^k \Ta \x                                                                                   \\
               &  & \st    & \sumai{+}1 \x + \sumai{+}2 \y - \sumai{-}1 \x - \sumai{-}2 \y = b_i^k &  & k \in K, i \in V,                 \\
               &  &        & \lmk_i - \lmk_j \leq \ca + \Ta                                        &  & k \in K, a \equiv (i, j) \in A_1, \\
               &  &        & \lmk_i - \lmk_j \leq \ca                                              &  & k \in K, a \equiv (i, j) \in A_2, \\
               &  &        & \suma1 (\ca+\Ta)\x + \suma2 \ca\y = \lmk_{o^k} - \lmk_{d^k}           &  & k \in K,                          \\
               &  &        & \x \geq 0                                                             &  & k \in K, a \in A_1,               \\
               &  &        & \y \geq 0                                                             &  & k \in K, a \in A_2,               \\
               &  &        & \Ta \geq 0                                                            &  & a \in A_1.
\end{align*}

This formulation has been referred to in most papers in the literature \cite{heilporn2006,didibiha2006,bouhtou2007,dewez2008},
including the first paper on the NPP \cite{brotcorne2001}. It only uses arc representations, so the need for path enumeration is eliminated.
Its linearized version (Section 2.3) only requires \(x\) to be integer. Furthermore, the number of arcs is fixed, hence,
it is reliable compared to the path-based formulations whose sizes depend on the number of paths of each commodity which is unknown.
Overall, the standard formulation is the most straightforward method to solve the NPP.

The second reformulation in which we are interested is the value function formulation (VF):
\begin{align*}
    \prog{VF} &  & \max\  & \sumk \suma1 \eta^k \Ta \x                                                                      \\
              &  & \st    & \sumai{+}1 \x + \sumai{+}2 \y - \sumai{-}1 \x - \sumai{-}2 \y = b_i^k &  & k \in K, i \in V,    \\
              &  &        & L^k \leq  \suma{} \dap\ca + \suma1 \dap\Ta                            &  & k \in K,  p \in \Pk, \\
              &  &        & (\ca+\Ta)\x + \suma2 \ca\y = L^k                                      &  & k \in K,             \\
              &  &        & \x \geq 0                                                             &  & k \in K, a \in A_1,  \\
              &  &        & \y \geq 0                                                             &  & k \in K, a \in A_2,  \\
              &  &        & \Ta \geq 0                                                            &  & a \in A_1.
\end{align*}

If we combine the second and the third constraints, a new constraint emerges:
\begin{equation}
    \suma1 (\ca + \Ta)\x + \suma2 \ca\y \leq \suma{} \dap\ca + \suma1 \dap\Ta \hspace{1cm} p \in \Pk. \label{vf}
\end{equation}

The left hand side of Eq. \eqref{vf} is the cost of the current path, while the right hand side is the cost of path \(p\).
Eq. \eqref{vf} means that the cost of the current path must not surpass the cost of any path in \(\Pk\), which restricts \((x, y)\)
to choose the path with least cost. This is called the value function constraint and it is also a popular method to generate the single-level formulation
of a bilevel problem. Here, we have shown that the value function method is just a special case of strong duality
when we write the dual representation of the follower problems in the solution space (dual-path in the case of the NPP).
The value function formulation for the NPP was previously mentioned in \cite{heilporn2006}.

The last formulation in which we focus is the path complementary slackness formulation (PCS):
\begin{align*}
    \prog{PCS} &  & \max\  & \sumk \suma1 \sump \eta^k \dap\Ta\zpk                                                   \\
               &  & \st    & \sump \zpk = 1                                                &  & k \in K,             \\
               &  &        & L^k \leq  \suma{} \dap\ca + \suma1 \dap\Ta                    &  & k \in K,  p \in \Pk, \\
               &  &        & \left(\suma{} \dap\ca + \suma1 \dap\Ta - L^k \right) \zpk = 0 &  & k \in K,  p \in \Pk, \\
               &  &        & \zpk \geq 0                                                   &  & k \in K, p \in \Pk,  \\
               &  &        & \Ta \geq 0                                                    &  & a \in A_1.
\end{align*}

This formulation is the path-based model of Didi-Biha et al. \cite{didibiha2006}.
It is the complete opposite of the standard formulation. It only uses path representations,
hence its performance heavily relies on the number of paths in \(\Pk\). Methods for reducing the size of \(\Pk\)
will be discussed in Section 3. In this formulation, only the selection of \(p\) is important and
any information on the structure of the graph is discarded. The advantage of this formulation is its simplicity, especially when
the size of \(\Pk\) is small.

\subsection{Linearization}

In all eight single-level formulations, there are bilinear terms (or in the case of VFCS, multilinear terms).
In order to solve these formulations using MILP solvers, these terms need to be linearized.
We will consider formulations using strong duality and those using complementary slackness separately.

\subsubsection{Linearization of Strong Duality Formulations}
In the strong duality formulations, the bilinear terms arise from the leader's revenue, which are either
\(\suma1 \Ta\x\) or \(\suma1 \sump \dap\Ta\zpk\). Because \(\x\) and \(\zpk\) can only take 0 and 1
as their values, we can force them to be binary and add new variables \(\Tak = \Ta\x = \Ta \sump \dap\zpk\) with the following constraints:
\begin{align}
     & 0 \leq \Tak \leq M_a^k \x           &  & a \in A_1, \label{directa1} \\
     & 0 \leq \Ta - \Tak \leq N_a (1 - \x) &  & a \in A_1 \label{directa2}
\end{align}
for the primal-arc representation and
\begin{align}
     & 0 \leq \Tak \leq M_a^k \sump \dap\zpk                       &  & a \in A_1, \label{directp1} \\
     & 0 \leq \Ta - \Tak \leq N_a \left(1 - \sump \dap\zpk \right) &  & a \in A_1 \label{directp2}
\end{align}
for the primal-path representation. \(M_a^k\) and \(N_a\) are big-M parameters.
We refer this method of linearization as \emph{direct linearization}.
Tight bounds for \(M_a^k\) and \(N_a\) can be found in Dewez et al. \cite{dewez2008}.
The leader revenue for each commodity becomes \(\suma1 \eta^k\Tak\), which we can use to replace the bilinear terms.
The strong duality constraints of the four linearized formulations are as follows:
\begin{align}
    \suma1 (\ca\x+\Tak) + \suma2 \ca\y      & = \lmk_{o^k} - \lmk_{d^k}, \label{lin-sd-aa} \\
    \suma1 (\ca\x+\Tak) + \suma2 \ca\y      & = L^k,                     \label{lin-sd-ap} \\
    \sump \suma{} \dap\ca\zpk + \suma1 \Tak & = \lmk_{o^k} - \lmk_{d^k}, \label{lin-sd-pa} \\
    \sump \suma{} \dap\ca\zpk + \suma1 \Tak & = L^k.                     \label{lin-sd-pp}
\end{align}

The objective function of the linearized formulation is independent from the primal representation:
\begin{align}
    \sumk \suma1 \eta^k \Tak.    \label{lin-obj-direct}
\end{align}

A summary of all linearized strong duality formulations is provided in Table \ref{tab:content-lin-sd}.
The variables with stars are required to be binary.

\begin{table}[h]
    \centering
    \begin{tabular}{lllllll}
        \toprule
              &                            &                        & \multicolumn{4}{c}{Constraints}                                                                \\
        \cmidrule(lr){4-7}
        Label & Variables                  & Obj.                   & Primal     & Dual                     & Opt. Cond.        & Linearization                      \\
        \midrule
        STD   & \(T, t, x^*, y, \lambda \) & \eqref{lin-obj-direct} & \eqref{pa} & \eqref{da1}, \eqref{da2} & \eqref{lin-sd-aa} & \eqref{directa1}, \eqref{directa2} \\
        VF    & \(T, t, x^*, y, L \)       & \eqref{lin-obj-direct} & \eqref{pa} & \eqref{dp}               & \eqref{lin-sd-ap} & \eqref{directa1}, \eqref{directa2} \\
        PASTD & \(T, t, z^*, \lambda \)    & \eqref{lin-obj-direct} & \eqref{pp} & \eqref{da1}, \eqref{da2} & \eqref{lin-sd-pa} & \eqref{directp1}, \eqref{directp2} \\
        PVF   & \(T, t, z^*, L \)          & \eqref{lin-obj-direct} & \eqref{pp} & \eqref{dp}               & \eqref{lin-sd-pp} & \eqref{directp1}, \eqref{directp2} \\
        \bottomrule
    \end{tabular}
    \caption{List of variables, constraints, and objective functions of all linearized strong duality formulations.}
    \label{tab:content-lin-sd}
\end{table}

\subsubsection{Linearization of Complementary Slackness Formulations}

In complementary slackness formulations, bilinear terms appear in all complementary slackness constraints.
In the arc-arc (CS) and path-path (PCS) combinations, each complementary slackness constraint is gated by a primal variable (\(x, y\), or \(z\)),
which can only take binary values. Thus, we could use them as branching condition for the constraints. Here are the linearized constraints
of the arc-arc combination (CS):
\begin{align}
    \lmk_i - \lmk_j & \geq \ca + \Ta - R_a^k(1 - \x) &  & a \equiv (i, j) \in A_1, \label{lin-cs-aa1} \\
    \lmk_i - \lmk_j & \geq \ca       - R_a^k(1 - \y) &  & a \equiv (i, j) \in A_2, \label{lin-cs-aa2}
\end{align}
and of the path-path combination (PCS):
\begin{align}
    L^k & \geq \suma{} \dap\ca + \suma1 \dap\Ta - S^k_{\pk}(1 - \zpk) &  & \pk \in \Pk. \label{lin-cs-pp}
\end{align}

\(R_a^k\) and \(S^k_{\pk}\) are big-M parameters depending on the dual representation
(\(R_a^k\) for dual-arc and \(S^k_{\pk}\) for dual-path).
When \(x, y\) or \(z\) are equal to 1, the big-M terms are dropped. Combining with the dual constraints,
these complementary slackness constraints force them to be active. If the primal variables are equal to 0,
the big-M terms make the constraints redundant, essentially removing them from the formulation.

To calculate the values of these big-M parameters,
recalling that \(N_a\) is the upper bound of \(\Ta\) defined in \eqref{directa2}, denote:
\begin{itemize}
    \item \(\underline{\lambda}^k_i\): the minimum cost from \(i\) to \(d^k\) when \(\Ta = 0, \forall a \in A_1\);
    \item \(\overline{\lambda}^k_j\): the minimum cost from \(j\) to \(d^k\) when \(\Ta = N_a, \forall a \in A_1\);
    \item \(\underline{L}^k\): the minimum cost from \(\ok\) to \(\dk\) when \(\Ta = 0, \forall a \in A_1\).
\end{itemize}

Then, the upper bounds for the big-M parameters are:
\begin{align*}
    R_a^k & = \ca + N_a - \underline{\lambda}^k_i + \overline{\lambda}^k_j &  & a \equiv (i, j) \in A_1, \\
    R_a^k & = \ca - \underline{\lambda}^k_i + \overline{\lambda}^k_j       &  & a \equiv (i, j) \in A_2, \\
    S^k_p & = \suma{} \dap\ca + \suma1 \dap N_a - \underline{L}^k          &  & p \in \Pk.
\end{align*}

In the calculation of \(\overline{\lambda}^k_j\), excluding all the tolled arcs (set \(T = \infty\)) is a valid option,
nevertheless, it may disconnect \(j\) and \(d^k\) which may lead to infinite \(R^K_a\).
Thus, since we aim for the smallest \(R^k_a\), we do not remove them.
A similar technique can be used for the path-arc combination (PACS):
\begin{align}
    \lmk_i - \lmk_j & \geq \ca + \Ta - R_a^k\left(1 - \sump \dap\zpk \right) &  & a \equiv (i, j) \in A_1, \label{lin-cs-pa1} \\
    \lmk_i - \lmk_j & \geq \ca       - R_a^k\left(1 - \sump \dap\zpk \right) &  & a \equiv (i, j) \in A_2. \label{lin-cs-pa2}
\end{align}

For the arc-path combination (VFCS), the constraints have multilinear terms. We will use the expression
\[\suma{} \dap - \suma1 \dap\x - \suma2 \dap\y\]

\noindent as the branching condition. The number of arcs in path \(p\) is \(\suma{} \dap\).
This expression is only equal to 0 if all the arcs along the path \(p\) are active, {\ie}
\(p\) is chosen. The linearized constraints of the arc-path combination (VFCS) are:
\begin{align}
    L^k & \geq \suma{} \dap\ca + \suma1 \dap\Ta - S^k_{\pk}\left(\suma{} \dap - \suma1 \dap\x - \suma2 \dap\y \right) &  & \pk \in \Pk. \label{lin-cs-ap}
\end{align}

After linearizing all the bilinear terms in the complementary slackness constraints, there are still bilinear terms
in the objective function (for the leader's revenue). We could use the direct linearization method as in the strong duality case
which adds more big-M constraints to the formulations. However, we can take advantage of the special structure of the NPP:
the leader's revenue appears twice, once in the objective function, and once in the strong duality constraint. Furthermore, since strong duality constraints are not utilized
in complementary slackness formulations, the leader's revenue can be extracted from these constraints and then be substituted in the objective function.
We will refer to this method as \emph{linearization by substitution}.
Let \(\tauk = \suma1 \Ta\x = \suma1 \sump \dap\Ta\zpk\) be the leader's revenue for commodity \(k\) per unit of demand.
The strong duality constraints are used to extract \(\tauk\):
\begin{align}
    \suma1 \ca\x + \suma2 \ca\y + \tauk & = \lmk_{o^k} - \lmk_{d^k}, \label{lin-subs-sd-aa} \\
    \suma1 \ca\x + \suma2 \ca\y + \tauk & = L^k,                     \label{lin-subs-sd-ap} \\
    \sump \suma{} \dap\ca\zpk + \tauk   & = \lmk_{o^k} - \lmk_{d^k}, \label{lin-subs-sd-pa} \\
    \sump \suma{} \dap\ca\zpk + \tauk   & = L^k.                     \label{lin-subs-sd-pp}
\end{align}

The objective function becomes:
\begin{align}
    \sumk \eta^k \tauk. \label{lin-obj-subs}
\end{align}

Table \ref{tab:content-lin-cs} is the summary of all linearized complementary slackness formulations.
The variables with stars are required to be binary. Since there are two linearization methods for the bilinear terms in the objective function,
a suffix is added to the label, where 1 means direct linearization and 2 means linearization by substitution.
Overall, the linearization of the complementary slackness formulations is more complicated than that of the strong duality formulations.
Complementary slackness formulations usually have more binary variables and more constraints, which may lead to worse performance.

\begin{table}[h]
    \centering
    \begin{tabular}{lllllll}
        \toprule
              &                                 &                        & \multicolumn{4}{c}{Constraints}                                                                                     \\
        \cmidrule(lr){4-7}
        Label & Variables                       & Obj.                   & Primal     & Dual                     & Opt. Cond.                             & Linearization                      \\
        \midrule
        \textbf{Direct linearization}\span\span\span                                                                                                                                           \\
        CS1   & \(T, t, x^*, y^*, \lambda \)    & \eqref{lin-obj-direct} & \eqref{pa} & \eqref{da1}, \eqref{da2} & \eqref{lin-cs-aa1}, \eqref{lin-cs-aa2} & \eqref{directa1}, \eqref{directa2} \\
        VFCS1 & \(T, t, x^*, y^*, L \)          & \eqref{lin-obj-direct} & \eqref{pa} & \eqref{dp}               & \eqref{lin-cs-ap}                      & \eqref{directa1}, \eqref{directa2} \\
        PACS1 & \(T, t, z^*, \lambda \)         & \eqref{lin-obj-direct} & \eqref{pp} & \eqref{da1}, \eqref{da2} & \eqref{lin-cs-pa1}, \eqref{lin-cs-pa2} & \eqref{directp1}, \eqref{directp2} \\
        PCS1  & \(T, t, z^*, L \)               & \eqref{lin-obj-direct} & \eqref{pp} & \eqref{dp}               & \eqref{lin-cs-pp}                      & \eqref{directp1}, \eqref{directp2} \\
        \midrule
        \textbf{Linearization by substitution}\span\span\span                                                                                                                                  \\
        CS2   & \(T, \tau, x^*, y^*, \lambda \) & \eqref{lin-obj-subs}   & \eqref{pa} & \eqref{da1}, \eqref{da2} & \eqref{lin-cs-aa1}, \eqref{lin-cs-aa2} & \eqref{lin-subs-sd-aa}             \\
        VFCS2 & \(T, \tau, x^*, y^*, L \)       & \eqref{lin-obj-subs}   & \eqref{pa} & \eqref{dp}               & \eqref{lin-cs-ap}                      & \eqref{lin-subs-sd-ap}             \\
        PACS2 & \(T, \tau, z^*, \lambda \)      & \eqref{lin-obj-subs}   & \eqref{pp} & \eqref{da1}, \eqref{da2} & \eqref{lin-cs-pa1}, \eqref{lin-cs-pa2} & \eqref{lin-subs-sd-pa}             \\
        PCS2  & \(T, \tau, z^*, L \)            & \eqref{lin-obj-subs}   & \eqref{pp} & \eqref{dp}               & \eqref{lin-cs-pp}                      & \eqref{lin-subs-sd-pp}             \\
        \bottomrule
    \end{tabular}
    \caption{List of variables, constraints, and objective functions of all linearized complementary slackness formulations.}
    \label{tab:content-lin-cs}
\end{table}

Combining with the four strong duality formulations, in total, there are 12 different MILP formulations for the NPP.
Some of them are completely new: (PASTD), (CS), and (VFCS), while
others are already mentioned in other works. Table \ref{tab:prevmodels} shows how these reformulations fit in the big picture.
In the left and the middle columns, the names and the labels used in the original works are listed.
The corresponding labels used in this paper are shown in the right column. We remark that in Bouhtou et al. \cite{bouhtou2007},
the path-based formulation (PMIP) is a path value function formulation (PVF) with a minor modification: the bilinear terms
are linearized separately per path by replacing \(\Tak = \Ta\x\) with \(r^k_{pa} = \Ta\x\zpk\).

\begin{table}
    \centering
    \begin{tabular}{lll}
        \toprule
        Name in original work                          & Original label & General label \\
        \midrule
        \textbf{Brotcorne et al. \cite{brotcorne2001}} &                                \\
        ~~~Mixed integer formulation                   & CPLEX          & STD           \\
        \textbf{Heilporn et al. \cite{heilporn2006}}   &                                \\
        ~~~Mixed integer formulation                   & TP2            & STD           \\
        ~~~New formulation                             & TP3            & VF            \\
        \textbf{Didi-Biha et al. \cite{didibiha2006}}  &                                \\
        ~~~Arc-based formulation                       & MIP I          & STD           \\
        ~~~Arc-path formulation                        & MIP II         & PACS1         \\
        ~~~Path-based formulation                      & MIP III        & PCS2          \\
        \textbf{Bouhtou et al. \cite{bouhtou2007}}     &                                \\
        ~~~Arc-based formulation                       & AMIP           & STD           \\
        ~~~Path-based formulation                      & PMIP           & PVF*          \\
        \textbf{Dewez et al. \cite{dewez2008}}         &                                \\
        ~~~Arc-based formulation                       & TOP-ARCS       & STD           \\
        ~~~Path-based formulation                      & PATH           & PCS2          \\
        \bottomrule
    \end{tabular}
    \caption{The mapping of previous formulations to the general framework.}
    \label{tab:prevmodels}
\end{table}

\section{Path Enumeration and Preprocessing}
This section describes a new path enumeration process in details, combining all the ingredients previously explored
in the literature. Then, we discuss the use of this process as a preprocessing method for arc-based formulations.
The latter will enable us to fairly compare arc and path-based reformulations, since we make preprocessing available for both.

\subsection{Path Enumeration}

Formulations using path representations require the set \(\Pk\), and thus,
the explicit enumeration of all paths from \(\ok\) to \(\dk\).
The size of \(\Pk\) can be exponential. However, not all paths are relevant.
We can eliminate many of them by using the dominance rule in Bouhtou et al. \cite{bouhtou2007}.

\begin{definition}
    \label{def:bifeas}
    Given a commodity \(k\), a path is bilevel feasible if it is optimal in the follower problem for some value of \(T\).
\end{definition}

Mathematically, if \(p\) is bilevel feasible, then there exists \(T\) such that for all \(q \in \Pk\):
\[\suma{} \dap\ca + \suma1 \dap\Ta \leq \suma{} \daq\ca + \suma1 \daq\Ta.\]

\begin{lemma}[Bouhtou et al. \cite{bouhtou2007}]
    \label{lem:dominance}
    Consider any commodity \(k \in K\). Let \(p\) and \(q\) be two different paths in \(\Pk\).
    If for all \(a \in A_1\), \(\dap \leq \daq\) and
    \[\suma{} \dap\ca < \suma{} \daq\ca,\]
    then \(q\) is not bilevel feasible,
    \ie, \(q\) cannot be the optimal path for any value of \(T\).
\end{lemma}

\begin{proof}
    Suppose that \(q\) is the optimal path for some fixed value of \(T\). The follower's cost
    of \(q\) must be minimal:
    \[\suma{} \dap\ca + \suma1 \dap\Ta \geq \suma{} \daq\ca + \suma1 \daq\Ta.\]
    However, because \(\dap \leq \daq\), this means \(\suma1 \dap\Ta \leq \suma1 \daq\Ta\);
    and because \(\suma{} \dap\ca < \suma{} \daq\ca\), this leads to a contradiction.
\end{proof}

The dominance rule in Lemma \ref{lem:dominance} is the only necessary rule to eliminate all non-bilevel-feasible paths.
Any remaining path is bilevel feasible.

\begin{theorem}
    \label{theo:bifeas}
    Any path which is not eliminated by the dominance rule in Lemma \ref{lem:dominance} is bilevel feasible.
\end{theorem}

\begin{proof}
    We will prove that if path \(p\) is not eliminated, then it is optimal for the following value of \(T\):
    \begin{equation*}
        T_a = \begin{cases}
            0      & \text{if } \dap = 1, \\
            \infty & \text{otherwise}.
        \end{cases}
    \end{equation*}
    Suppose that \(p\) is not optimal for the above value of \(T\) and let \(q\) be the optimal path.
    Consequently, \(q\) can only use arcs with \(\dap = 1\) because all other tolled arcs are disabled.
    In other words, \(\daq \leq \dap\).
    Also because \(T_a = 0\) for \(\dap = 1\), the costs of \(p\) and \(q\)
    in this case are exactly their initial costs, hence \(\suma{} \daq\ca < \suma{} \dap\ca\).
    By Lemma \ref{lem:dominance}, \(p\) should have been eliminated which is a contradiction.
    Therefore, \(p\) must be optimal for the given value of \(T\).
\end{proof}

Definition \ref{def:bifeas} alone does not ensure that any bilevel feasible path is relevant to solve the NPP.
Suppose we have two different paths with the same initial cost and the same set of tolled arcs.
Then they can be both bilevel feasible. In such case, however, only one path is required.
Thus, we employ Assumption \ref{as:relevant} to make the term ``bilevel feasible'' coincide with the term ``relevant''.

\begin{assumption}
    \label{as:relevant}
    Two different bilevel feasible paths must have different initial costs.
\end{assumption}

This assumption also implies that two different bilevel feasible paths will have different sets of tolled arcs.
This can be fulfilled if we add a minuscule random perturbation to each arc.

Yen's algorithm \cite{yen1971} is usually used to enumerate the set of paths of a commodity.
The algorithm outputs the shortest path, the second shortest path, etc, up to the \(K\)-th shortest path between two nodes in the graph.
In this case, the shortest path is the path with the minimum initial cost (the cost when set \(T = 0\)).
The algorithm should be stopped when it has found the first toll-free path as in the following corollary:

\begin{corollary}
    \label{cor:stop}
    Given a commodity \(k \in K\), a path \(p\) cannot be bilevel feasible if its initial cost is greater than that of the shortest toll-free path \(\pik\).
\end{corollary}

In Bouhtou et al. \cite{bouhtou2007}, the authors reduced the size of the graph by transforming the original graph to the shortest path graph model (SPGM).
In SPGM, all nodes that are not incident to any tolled arc will be removed and the toll-free arcs connecting to it are replaced by
the outer product of the list of incoming arcs and the list of outgoing arcs. Then, some elimination rules (which are only applicable on the SPGM)
are applied to reduce the number of arcs. Finally, the set of paths is enumerated using the Yen's algorithm on the SPGM and the dominance rule is applied afterward.
A drawback of the SPGM is the inflation of the set of arcs which is caused by the replacement of nodes by arcs.
Besides that, the elimination rules in SPGM still leave many redundant nodes and arcs. Thus, there is a sizable number of irrelevant
paths generated by the Yen's algorithm.

In Didi-Biha et al. \cite{didibiha2006}, the authors used a modified version of the Lawler's procedure \cite{lawler1972}, a general form of the Yen's algorithm, to enumerate the set of bilevel feasible paths directly without using the SPGM.
SPGM is not needed here because it is only a method to reduce the size of the graph before Yen's algorithm is applied which is not used by the authors.
Due to Assumption \ref{as:relevant}, only paths with different combinations of tolled arcs are explored,
ignoring most redundant paths. However, the algorithm in \cite{didibiha2006} requires solving a shortest path problem with some arcs fixed to 1
which cannot be accomplished by the Dijsktra's algorithm and requires a linear program solver.
We propose an improved version of this algorithm which just needs a shortest path algorithm.
Algorithm \ref{alg:enum} shows this enumeration process in details. In the algorithm,
\(C\) is the set of candidate paths, \(N\) is the maximum number of paths we want to enumerate,
\(\psj\) is the \(j\)-th shortest path, \(s(\qi)\) is the spur node of path \(\qi\),
and \(R(\qi)\) is the set of excluded tolled arcs of path \(\qi\).

\begin{algorithm}
    \caption{Path enumeration}
    \label{alg:enum}
    \hspace*{\algorithmicindent}\textbf{Input}: The graph \(G = (V, A)\), the O-D pair \(o^k, d^k\), the initial costs \(c_a\), the maximum number of paths to be enumerated \(N\).

    \hspace*{\algorithmicindent}\textbf{Output}: A list of \(N\) paths, most of which are bilevel feasible.

    \begin{algorithmic}
        \STATE Find the first shortest path \(\ps1\)
        \STATE \(C \gets \{\ps1\}\)
        \STATE \(s(\ps1) \gets \ok\)
        \STATE \(R(\ps1) \gets \varnothing\)
        \STATE \(j \gets 1\)
        \WHILE{$j \leq N$}
        \STATE Output the path with minimum cost in \(C\), call it \(\psj\)
        \STATE Remove \(\psj\) from \(C\)
        \IF{$\psj$ has no tolled arcs}
        \STATE Stop the algorithm
        \ENDIF
        \STATE Let \(a_1, a_2, \ldots, a_m\) be the ordered tolled arcs of \(\psj\) from \(s(\psj)\) to \(\dk\)
        \STATE \(\hat{s} \gets s(\psj)\)
        \FOR{$i$ from 1 to $m$}
        \STATE \(\qi \gets \) subpath of \(\psj\) from \(\ok\) to \(\hat{s}\)
        \STATE \(s(\qi) \gets \hat{s}\)
        \STATE \(R(\qi) \gets R(\psj) \cup \{a_i\}\)
        \STATE Remove nodes in \(\qi\) (for this loop only)
        \STATE Remove arcs in \(R(\qi)\) (for this loop only)
        \IF{$\hat{s}$ and $\dk$ are connected}
        \STATE Append to \(\qi\) the shortest path from \(\hat{s}\) to \(\dk\)
        \STATE Add \(\qi\) to \(C\)
        \ENDIF
        \STATE \(\hat{s} \gets\) target of \(a_i\)
        \ENDFOR
        \STATE \(j \gets j + 1\)
        \ENDWHILE
    \end{algorithmic}
\end{algorithm}

Algorithm \ref{alg:enum} is a version of the Lawler's algorithm \cite{lawler1972}
which is an enumeration scheme for binary problems such as the shortest path problem.
Let \(S(\psj)\) be the set of tolled arcs of \(\psj\) from \(\ok\) to \(s(\psj)\).
We can see that \(\psj\) is a solution (not necessarily optimal) of the shortest path problem
when all arcs in \(R(\psj)\) are fixed to 0 and all arcs in \(S(\psj)\) are fixed to 1.
We will call \(\psj\) a solution of the problem \((R(\psj), S(\psj))\).
The Lawler's algorithm tells us that we need to spawn \(n\) subproblems where \(n\) is the number of non-fixed variables.
We can re-order these non-fixed variables in any order. Hence, the variables with value 1 (\(x_{a_1}\) to \(x_{a_m}\)) come first
and the variables with value 0 (\(x_{a_{m+1}}\) to \(x_{a_n}\)) follow.
The former \(m\) variables are the tolled arcs of \(\psj\) from \(s(\psj)\) to \(\dk\)
(all the tolled arcs preceding \(s(\psj)\) are fixed by assumption).
These \(m\) variables must also be in order of appearance as mentioned in Algorithm \ref{alg:enum}.
According to this order, the \(n\) subproblems of the Lawler's algorithm are generated
by fixing all the previously fixed variables with these additional constraints (one constraint for each subproblem):
\begin{align*}
    (1)     &  &  & x_{a_1} = 0                                                                         \\
    (2)     &  &  & x_{a_1} = 1, x_{a_2} = 0                                                            \\
    \vdots~ &  &  & \hspace{1cm} \vdots                                                                 \\
    (m)     &  &  & x_{a_1} = x_{a_2} = \ldots = x_{a_{m-1}} = 1, x_{a_m} = 0                           \\
            &  &  &                                                                                     \\
    (m+1)   &  &  & x_{a_1} = \ldots = x_{a_m} = 1; x_{a_{m+1}} = 1                                     \\
    (m+2)   &  &  & x_{a_1} = \ldots = x_{a_m} = 1; x_{a_{m+1}} = 0, x_{a_{m+2}} = 1                    \\
    \vdots~ &  &  & \hspace{1cm} \vdots                                                                 \\
    (n)     &  &  & x_{a_1} = \ldots = x_{a_m} = 1; x_{a_{m+1}} = \ldots = x_{a_{n-1}} = 0, x_{a_n} = 1
\end{align*}

We can verify that the feasible regions of all subproblems are mutually exclusive and their union
is equivalent to the feasible region of the problem \((R(\psj), S(\psj))\) except \(\psj\).
The subproblems in the second group have a common constraint \(x_{a_1} = \ldots = x_{a_m} = 1\).
However, because of Lemma \ref{lem:dominance}, \(x_{a_1} = \ldots = x_{a_m} = 1\)
implies that the resulting paths cannot be bilevel feasible (dominated by \(\psj\)).
Therefore, we only need to consider the first \(m\) subproblems.

Consider the subproblem \(i\). Let \(\qi\) be the path constructed as in Algorithm \ref{alg:enum}.
We call \(\qi\) a child of \(\psj\). There are two cases: we can construct \(\qi\) and we cannot (there is no path from \(\hat{s}\) to \(\dk\)).
If we can and \(\qi\) exists, then similar to \(\psj\), \(\qi\) is a solution of the problem \((R(\qi), S(\qi))\) where
\(R(\qi) = R(\psj) \cup \{a_i\}\) and \(S(\qi) = S(\psj) \cup \{a_1, \ldots, a_{i-1}\}\).
Although we do not assume that \(\psj\) and \(\qi\) are optimal in their own subproblems,
they satisfy two special properties. First, the costs of \(\psj\) and its child \(\qi\) follow the correct enumeration order:

\begin{lemma}
    \label{lem:order}
    The cost of \(\qi\) is at least the cost of \(\psj\).
\end{lemma}

Hereafter, we will use the term \emph{part} to refer to a subpath of \(\psj\) or \(\qi\) of which one end is either the origin \(\ok\)
or the destination \(\dk\). The term \emph{segment} will refer to any other subpath.

\begin{proof}
    Because \(\psj\) is constructed by Algorithm \ref{alg:enum}, its part from \(s(\psj)\) to \(\dk\)
    is the shortest path excluding arcs in \(R(\psj)\) and all nodes preceding \(s(\psj)\).
    As a result, its part from \(s(\qi)\) to \(\dk\) is also the shortest path with respect to the same set of constraints.
    Now, consider the path \(\qi\). Its part from \(\ok\) to \(s(\qi)\)
    is identical to that of \(\psj\). The other part from \(s(\qi)\) to \(\dk\)
    is the shortest path, while excluding arcs in \(R(\qi)\) and all nodes preceding \(s(\qi)\).
    Given that \(R(\psj) \subset R(\qi)\) and \(s(\psj)\) comes before \(s(\qi)\), the new shortest path problem is a restriction.
    Thus, the cost of this part of \(\qi\) must be at least the cost of the same part of \(\psj\).
\end{proof}

Lemma \ref{lem:order} is fundamental for using Corollary \ref{cor:stop} as the stopping condition.
If the enumeration order is not maintained, then we may stop prematurely after encountering the first toll-free path.
The other property is expressed in the lemma below and it will be used for proving the correction of another pruning condition.

\begin{lemma}
    \label{lem:opttollfree}
    Any toll-free segment of \(\qi\) (or any path generated by the algorithm) is the shortest segment
    while excluding all arcs in \(R(\qi)\) and all the preceding nodes of that segment.
\end{lemma}

\begin{proof}
    Assume that \(\qi\) is a child of \(\psj\) and the lemma is true for \(\psj\). The part of \(\psj\)
    from \(\ok\) to \(s(\qi)\) is reused for \(\qi\) and does not contain any arc in \(R(\qi)\), so the lemma is true for this part.
    Consider the other part of \(\qi\) from \(s(\qi)\) to \(\dk\).
    By construction in Algorithm \ref{alg:enum}, this part is the shortest path, while excluding arcs in \(R(\qi)\) and nodes preceding \(s(\qi)\).
    As a result, any subpath of this part is also the shortest subpath with respect to the same set of constraints.
    Later, for any toll-free segment in this part, we can extent the set of excluded nodes from the set of nodes that precedes \(s(\qi)\)
    to the set of nodes that precedes that segment. This means that the lemma is also true for this part of \(\qi\).
    Recall that \(s(\qi)\) is the target of some tolled arc, thus there are no toll-free segments containing \(s(\qi)\) in the middle
    and the lemma is true for \(\qi\). To finish the induction proof, we consider the base case of the first shortest path for which the lemma is also true.
\end{proof}

Now, we consider the other case where we cannot construct \(\qi\). This means
\(s(\qi)\) and \(\dk\) are disconnected when we remove arcs in \(R(\qi)\) and nodes preceding \(s(\qi)\).
We will prove that there is no bilevel feasible path in \((R(\qi), S(\qi))\):

\begin{lemma}
    \label{lem:prune}
    If \(\qi\) does not exists, any path in \((R(\qi), S(\qi))\) is not bilevel feasible.
\end{lemma}

\begin{proof}
    If the problem \((R(\qi), S(\qi))\) is infeasible, then the above statement is trivially true. Suppose it
    has a feasible solution \(\hqi\) which is a simple path (path with no loops).
    We will construct another path (not necessarily a solution of \((R(\qi), S(\qi))\))
    which dominates \(\hqi\) by replacing a segment in \(\hqi\) by a shorter segment in \(\psj\). First, we need to
    sort the arcs in \(S(\qi)\) by the order of appearance in \(\psj\):
    \begin{equation*}
        \ha_1, \ha_2, \ha_3, \ldots, \ha_{n-1}, \ha_n
    \end{equation*}
    where \(n = |S(\qi)|\). We also define \(\ha_0\) as a virtual arc whose target is \(\ok\) (its source is not relevant).
    Next, for \(0 \leq i < n\), we will try to replace the segment of \(\hqi\) between \(\ha_i\) and \(\ha_{i+1}\) with
    the same segment of \(\psj\) if \(\ha_{i+1}\) comes after \(\ha_i\) in \(\hqi\) and if the latter segment has smaller cost.
    If we are able to make such a replacement, due to the definition of \(S(\qi)\),
    the segment connecting \(\ha_i\) and \(\ha_{i+1}\) in \(\psj\) is toll-free; hence, we only remove tolled arcs while building a better path,
    and thus the resulting path dominates \(\hqi\). If we cannot replace any segment, this means that either all segments in \(\hqi\) are identical to those in \(\psj\)
    or there is a segment in \(\hqi\) with smaller cost.

    Consider the first case: all segments in \(\hqi\) are identical to those in \(\psj\). In this case, the parts of \(\hqi\) and \(\psj\)
    from \(\ok\) to \(s(\qi)\) are identical. The other part of \(\hqi\) must be a path from \(s(\qi)\) to \(\dk\). Because we assume
    that \(\hqi\) is a simple path, we must exclude all nodes preceding \(s(\qi)\). However, since we also assume that \(\qi\)
    does not exist, this implies \(s(\qi)\) and \(\dk\) to be disconnected if we exclude those nodes. This is a contradiction.

    Consider the second case: there is a segment in \(\hqi\) with smaller cost. Let \(\ha_m\) be the first arc such that the segment of \(\hqi\)
    from \(\ha_m\) to \(\ha_{m+1}\) is cheaper than its counterpart of \(\psj\). All segments preceding \(\ha_m\) in both \(\hqi\) and \(\psj\) must be identical
    since we assume that we cannot replace any segment. Consequently, the part from \(\ok\) to \(\ha_m\) of both paths are the same.
    By Lemma \ref{lem:opttollfree}, the segment of \(\psj\) from \(\ha_m\) to \(\ha_{m+1}\) is the optimal path while excluding all preceding nodes.
    However, the same segment of \(\hqi\) is better, which means it must violate that constraint and must repeat some node preceding \(\ha_m\).
    This is again a contradiction as we assumed that \(\hqi\) is a simple path.

    Since both cases result in a contradiction, we can always replace some segment in \(\hqi\) with a better segment in \(\psj\).
    Therefore, \(\hqi\) cannot be bilevel feasible.
\end{proof}

\begin{theorem}
    Given a sufficiently large value of \(N\), the output of Algorithm \ref{alg:enum} includes all bilevel feasible paths.
\end{theorem}

\begin{proof}
    At any step of the algorithm, the subproblems generated by the Lawler's procedure
    do not remove any part of the feasible region of the original problem.
    The subproblems are mutually exclusive and there is a finite number of paths,
    thus the algorithm is guaranteed to stop in finite time.
    A subproblem \(i\) is pruned if and only if it is in the second group of subproblems or if \(\qi\) does not exists.
    Both cases imply that all solutions in the feasible region of the subproblem are not bilevel feasible (Lemma \ref{lem:prune}).
    Besides that, the stopping condition in Corollary \ref{cor:stop} is ensured by Lemma \ref{lem:order}.
    Therefore, no bilevel feasible solution is ruled out.
\end{proof}

Although the set of paths enumerated in Algorithm \ref{alg:enum} is close to the final set of bilevel feasible paths,
there are still redundant paths. To obtain the final set of bilevel feasible paths, the dominance rule in Lemma \ref{lem:dominance}
still needs to be applied once more.

\begin{example}
    Consider the graph in Figure \ref{fig:counter}. The number of each arc represents the initial cost \(\ca\).
    The first shortest path is \(\ok - u - v - \dk\) with the cost of 3. By Algorithm \ref{alg:enum}, three subproblems are generated, each producing a candidate path.
    They are the second shortest path \(\ok - u - \dk\) (cost 4), the third shortest path \(\ok - u - v - w - \dk\) (cost 6), and
    the toll-free path \(\ok - \dk\) (cost 10). All 4 paths will be returned, however, the path \(\ok - u - \dk\) dominates \(\ok - u - v - w - \dk\).
    The final set of bilevel feasible paths only has 3 paths: \(\ok - u - v - \dk\), \(\ok - u - \dk\), and \(\ok - \dk\).
\end{example}

\begin{figure}
    \centering
    \begin{tikzpicture}[scale=2]
        \tikzstyle{vertex}=[circle,draw,minimum size=20pt,inner sep=0pt]
        \tikzstyle{tolled}=[->, >=latex, dashed]
        \tikzstyle{tollfree}=[->, >=latex]

        \node[vertex] (o) at (0, 0) {$\ok$};
        \node[vertex] (u) at (1, 0) {$u$};
        \node[vertex] (v) at (2, 0) {$v$};
        \node[vertex] (w) at (2.5, -0.8) {$w$};
        \node[vertex] (d) at (3, 0) {$\dk$};

        \draw[tolled] (o) edge node[below]{1} (u) (u) edge node[below]{1} (v) (v) edge node[below]{1} (d);
        \draw[tollfree] (v) edge node[left]{2} (w) (w) edge node[right]{2} (d);
        \draw[tollfree, bend left=45] (o) edge node[above]{10} (d);
        \draw[tollfree, bend left=30] (u) edge node[above]{3} (d);
    \end{tikzpicture}
    \caption{Graph with redundant paths (dashed arcs are tolled arcs).}
    \label{fig:counter}
\end{figure}
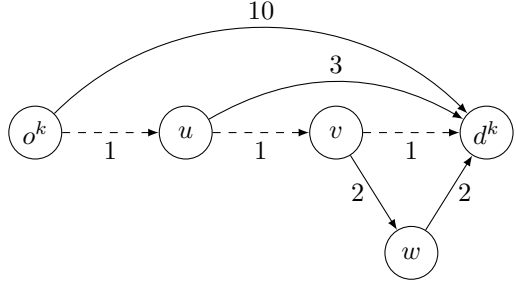

\subsection{Path-based Preprocessing}
In \cite{didibiha2006, bouhtou2007}, the authors compare their proposed path-based formulations
to an unprocessed or SPGM-processed standard formulation. In these comparisons,
path-based formulations are more advantageous compared to their counterparts.
Although SPGM can reduce the complexity of a graph to some degree,
path-based formulations essentially bypass all the redundant arcs and nodes in the graph,
hence they have more preprocessing power.
However, we could harness this preprocessing power of path enumeration and apply it to arc-based formulations.
We propose a new preprocessing method which uses the set of enumerated paths to remove all redundant arcs and nodes completely.
The idea is simple: given the set of bilevel feasible paths, any arc or node which does not appear in any path will be removed.
Then, to further reduce the size of the graph, chains of toll-free arcs will be replaced with a single toll-free arc with cost
equal to the sum of the costs of the associated arcs.
This preprocessing method will let the arc-based formulations use the information generated from the
path enumeration process and enable us to compare all the formulations in a fair manner.

In addition, SPGM can still be applied on top of the processed graph. This does not remove any additional tolled arcs
because the set of tolled arcs obtained from path-based preprocessing is already minimal. However, SPGM transforms
the graph by replacing nodes with toll-free arcs, which can sometimes reduce both the number of nodes and the number of toll-free arcs
in sparse graphs.

As mentioned in Section 2.3, the conversions used in the (PACS) and the (VFCS) formulations must cover all solutions in the primal representation.
When we apply the path-based preprocessing to the arc representation and use the set of bilevel feasible paths for
the path representation, the (PACS) formulation still satisfies this rule. This is because the set of primal solutions of (PACS)
is the set of bilevel feasible paths, all of them can be converted to arcs in the processed graph.
On the other hand, the (VFCS) formulation does not satisfy this rule anymore. In the processed graph, there are many paths
which are not bilevel feasible. If these paths are not covered in the (VFCS) formulation, the optimality condition can be bypassed
by selecting these paths as the primal solutions. However, enumerating all paths (including non-bilevel-feasible) in the processed graph would be too costly.
Therefore, a special treatment is neccesary and we propose a cutting-plane method just for (VFCS):
the processed graph and the set of bilevel feasible paths will still be used to generate the initial set of constraints.
Then, when we encounter a solution, we will examine if the path in that solution is a bilevel feasible path.
If it is not, we will use it to generate an additional complementary slackness constraint to cover the new path and continue solving.

\section{Hybrid Model for Multi-Commodity Problems}

Although path-based formulations and preprocessing can give the MIP solver a boost by removing redundant variables,
they require time to enumerate the paths. Since the number of paths can be exponential,
sometimes, it is faster to just solve the problem without relying on path enumeration (unprocessed or SPGM-processed STD or CS).
However, it is unknown which option is better until the enumeration process is finished.
Instead of a full enumeration, we could enumerate until a predetermined number of paths and then
decide if it is worth to continue enumerating or to cancel the enumeration process and fallback to a formulation
which does not require path enumeration. Since we have a multi-commodity problem,
this probing method can be applied for each commodity separately. The end result is a hybrid model,
assigning different formulations and preprocessing for different commodities depending on their numbers of paths.
This process is shown in Algorithm \ref{alg:hybrid}. In the algorithm, \(N\) is called the breakpoint, which
is a threshold to decide whether the enumeration process should be continued. The commodities with only 1 path
are redundant, because the only path must be a toll-free path which does not bring any profit to the leader.
Commodities with less than \(N\) paths are assigned to a path-based formulation or an arc-based formulation
with path preprocessing, while commodities with more than \(N\) paths are assigned to a fallback formulation
which does not need path enumeration. All variables are managed independently according to their corresponding commodities,
except for \(\Ta\) which is shared.

\begin{algorithm}
    \caption{Hybrid model}
    \label{alg:hybrid}
    \hspace*{\algorithmicindent}\textbf{Input}: The graph \(G = (V, A)\) with costs \(\ca\), all commodities \((\eta^k, \ok, \dk)\).

    \hspace*{\algorithmicindent}\textbf{Output}: A hybrid model \((obj, constraints)\).

    \begin{algorithmic}
        \STATE \(obj \gets 0\)
        \STATE \(constraints \gets \varnothing\)
        \FOR{$k \in K$}
        \STATE Enumerate the first \(N + 1\) paths of commodity \(k\)
        \STATE Let \(\hat{P}^k\) be the set of enumerated paths
        \IF{$|\hat{P}^k| > 1$}
        \IF{$|\hat{P}^k| \leq N$}
        \STATE Assign a formulation \(F^k\) which is either path-based or arc-based with path-based preprocessing
        \ELSE
        \STATE Assign an arc-based formulation \(F^k\) without path-based preprocessing
        \ENDIF
        \STATE \(obj \gets obj + obj(F^k)\)
        \STATE \(constraints \gets constraints \cup constraints(F^k)\)
        \ENDIF
        \ENDFOR
        \STATE Use a solver to solve \((obj, constraints)\)
    \end{algorithmic}
\end{algorithm}

As an example, suppose we divide the set of commodities \(K\) into 3 parts based on the number of paths: \(K_1\) for
commodities having only 1 path, \(K_2\) for commodities with less than \(N\) paths, and \(K_3\) for
commodities with more than \(N\) paths. If we assign (PCS2) to \(K_2\) and (STD) to \(K_3\), then we get the following hybrid model:

\begin{align*}
    \max\  & \sum_{k \in K_2} \eta^k \tau^k + \sum_{k \in K_3} \suma1 \eta^k \Tak                                  \\
    \st    & \eqref{pp}, \eqref{dp}                                                   &  & k \in K_2,              \\
           & L^k \geq \suma{} \dap\ca + \suma1 \dap\Ta - S^k_{\pk}(1 - \zpk)          &  & k \in K_2, \pk \in \Pk, \\
           & \sump \suma{} \dap\ca\zpk + \tau^k = L^k                                 &  & k \in K_2,              \\
           & \zpk \in \{0, 1\}                                                        &  & k \in K_2, \pk \in \Pk, \\
    \\
           & \eqref{pa}, \eqref{da1}, \eqref{da2}, \eqref{directa1}, \eqref{directa2} &  & k \in K_3,              \\
           & \suma1 (\ca\x+\Tak) + \suma2 \ca\y = \lmk_{o^k} - \lmk_{d^k}             &  & k \in K_3,              \\
           & \x \in \{0, 1\}                                                          &  & k \in K_3, a \in A_1.
\end{align*}

The hybrid model can be extended to have multiple breakpoints and different assignment schemes.
For example, path-based formulations are less effective for commodities which have many paths
because the more paths a commodity has, the more variables and constraints are added.
In contrast, arc-based formulations with path preprocessing only remove arcs and nodes
which guarantee complexity reduction. We could use the hybrid model to balance between these two kinds of formulations,
by assigning path-based formulations to commodities with few paths
and arc-based formulations for those with many paths.

\section{Experiments}
We conducted computational experiments to  evaluate in practice the reformulations
and preprocessing presented in the previous sections.
In Section 5.1, the experimental methodology is described,
including the implementation details, the generation and properties of the instances used.
Section 5.2 provides a comparison between the 12 reformulations listed in Section 2.
Section 5.3 validates the efficiency of the new path-based preprocessing.
In Section 5.4, we justify the use of the hybrid model introduced in Section 4.

\subsection{Methodology}
The instances used in the tests is randomly generated. We employ a generation method similar to the one described in \cite{brotcorne2000}
to make the problems challenging. For a given graph and a set of O-D pairs, first, the shortest path of each commodity is found.
Next, we count the number of paths passing through each arc and sort them in the descending order of that count.
Following that order, each arc is converted into a tolled arc until 2/3 of the desired number of tolled arcs is reached.
The last 1/3 is selected randomly among all remaining arcs.
An arc is converted only if it does not remove the last toll-free path for all commodities.
To make the generated data more realistic, we also enforce the properties of the arc to be symmetrical,
which means that any arc and its reversed arc will have the same cost, and both must be either tolled or toll-free.
The cost of 80\% of all arcs will be distributed uniformly from 5 to 35, while the remaining 20\% will have the maximum cost of 35.
The cost of tolled arcs are halved after the conversion. The proportion of tolled arcs is 20\%.

There are 200 generated problem instances (or in short, problems) divided into four sets which are summarized in Table \ref{tab:dataclass}.
Each set has 50 problems with 5 different numbers of commodities: 30, 35, 40, 45, and 50 (10 problems generated for each value).
The topologies are inspired from \cite{brotcorne2008} and are illustrated in Figure \ref{fig:topo}.
The set G provides a dataset similar to the one used in \cite{brotcorne2000}
and \cite{brotcorne2001}, while the three other sets have more paths and are more challenging to solve.
The cumulative distribution of the number of bilevel feasible paths of each set is shown in Figure \ref{fig:numpathsdata}.
The horizontal axis represents the breakpoint \(N\) while the vertical axis shows the proportion of commodities
with no more than \(N\) bilevel feasible paths.

\begin{table}
    \centering
    \begin{tabular}{lllrr}
        \toprule
        Label & Topology & Dimensions             & \(|V|\) & Avg. \(|A|\) \\
        \midrule
        G     & Grid     & \(5 \times 12\) nodes  & 60      & 206          \\
        H     & Grid     & \(12 \times 12\) nodes & 144     & 528          \\
        D     & Delaunay & -                      & 144     & 832          \\
        V     & Voronoi  & -                      & 144     & 410          \\
        \bottomrule
    \end{tabular}
    \caption{Properties of generated data.}
    \label{tab:dataclass}
\end{table}

\begin{figure}
    \centering
    \input{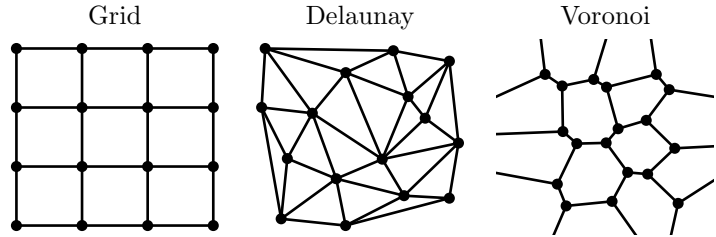}
    \caption{Illustration of the three types of topology.}
    \label{fig:topo}
\end{figure}

\begin{figure}
    \centering
    \input{plots/numpaths.pgf}
    \caption{Cumulative distribution of the number of bilevel feasible paths.}
    \label{fig:numpathsdata}
\end{figure}
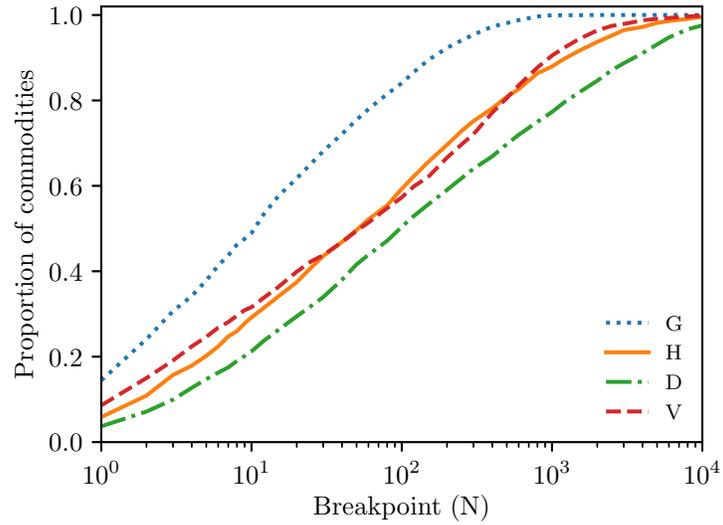

Our code is implemented in C++.
All test runs are executed single-threaded on the B{\'e}luga cluster (Intel Xeon 2.4 GHz) under Linux.
All tested models are instantiations of the hybrid model described in Algorithm \ref{alg:hybrid}.
Given a problem, the path enumeration will be applied first, then
a formulation will be assigned to each commodity and the single-level reformulation
is synthesized. Finally, the single-level reformulation is solved directly by CPLEX 12.9.
If a model does not rely on path enumeration, the breakpoint \(N\) is set to 1 and path enumeration
is still applied. This will remove the commodities with only 1 path and provide basic preprocessing for all models.
The main formulation will be assigned for the commodities with less than \(N\) bilevel feasible paths.
The default fallback formulation for commodities with more than \(N\) paths is the (STD) with no preprocessing.
By increasing \(N\), we essentially replace the fallback formulation with the main formulation. If the main formulation is better,
then we should see an improvement in performance when increasing \(N\) and vice versa.
Some formulations are only good up to a certain breakpoint,
which can also be observed if there is a change in the direction of their performance over \(N\).
	
We imposed a time limit of 1 hour for each problem.
The time for path enumeration is included in the total time of a test run.
If the run exceeds 1 hour, it is stopped and the optimality gap is recorded.
The problems are divided into two groups: easy and hard. The easy group consists of 108
problems which can be solved by at least one run. All the remaining 92 problems form the hard group, for which an optimal solution is not available. When we compare the performance of any two models,
we are interested in three criteria: the number of problems solved by each model, the average time it takes to solve the easy group,
and the average optimality gap of the the hard group.

\subsection{Comparison of all Formulations}
In this section, we compare the performance of all 12 formulations.
These will be assigned as the main formulations of the hybrid models.
If the main formulation is arc-based, then path-based preprocessing is applied.
If the main formulation is a mixed arc-path or path-arc formulation, then we use path-based preprocessing
on the arc representation. If the main formulation is a path-based formulation,
no further preprocessing is applied. In all cases, path enumeration is required for the main formulations.

Table \ref{tab:models} shows the performance of all models over \(N\).
At first look, we can see that formulations with suffix 1 worsen over \(N\).
This means they are even worse than the fallback formulation.
These are the formulations with complementary slackness as optimality condition and direct linearization.
The reason may be because both complementary slackness and direct linearization use big-M constraints.
All other formulations perform roughly the same (except for (VFCS2)),
with the standard formulation taking the lead in all three criteria.

Figure \ref{fig:models} plots the performance over \(N\) of four notable formulations: (STD), (VF), (CS2), and (PCS2).
(STD), (VF), and (CS2) represent three basic paradigms to convert a bilevel linear problem to a single-level reformulation.
(PCS2) is a path-based formulation which is mentioned in \cite{didibiha2006} and \cite{dewez2008}. In these papers, their experiments suggest that
(PCS2) outperforms (STD), but this comparison is done when the (STD) formulation is preprocessed by the SPGM method.
Here, we observe the opposite result: the standard formulation outperformed (PCS2). The key is the new path-based preprocessing, which
theoretically provides arc-based formulations with the same preprocessing power of the path enumeration, thus it levels the playing field.
Figure \ref{fig:models} also shows a drawback of path-based formulations: the larger \(N\) is, the more complex they become.
(PCS2) must add an extra binary variable for every enumerated path, while (STD) and (CS2) use the set of paths to eliminate arcs and nodes.
Thus, the complexity of (PCS2) increases over \(N\) while the complexity of the other two formulations is capped by its unprocessed version.
In the graph, we can see that the performance of (PCS2) starts dropping after \(N = 500\).
The other formulations using primal-path representation ((PASTD), (PVF), (PACS)) also suffer from this problem.
(VF) uses dual-path, so it only adds more constraints instead of binary variables which allows it to continue improving over \(N\).

\begin{table}
    \small
    \centering
    \begin{subtable}[h]{\textwidth}
        \centering
        \begin{tabular}{lrrrrrrrrrr}
            \toprule
            \(N\) & 10 & 20 & 50 & 100 & 200 & 500 & 1000 \\
            \midrule
            std   & 61 & 64 & 69 & 72  & 78  & 92  & 94   \\
            vf    & 59 & 64 & 68 & 68  & 75  & 90  & 93   \\
            pastd & 56 & 64 & 68 & 71  & 76  & 86  & 85   \\
            pvf   & 59 & 64 & 67 & 69  & 79  & 80  & 81   \\
            cs1   & 47 & 32 & 8  & 1   & 0   & 0   & 0    \\
            cs2   & 58 & 62 & 67 & 67  & 70  & 71  & 73   \\
            pacs1 & 49 & 33 & 8  & 1   & 0   & 0   & 0    \\
            pacs2 & 58 & 62 & 63 & 68  & 70  & 76  & 69   \\
            vfcs1 & 43 & 10 & 0  & 0   & 0   & 0   & 0    \\
            vfcs2 & 56 & 55 & 44 & 36  & 29  & 23  & 21   \\
            pcs1  & 48 & 33 & 9  & 2   & 0   & 0   & 0    \\
            pcs2  & 61 & 61 & 72 & 71  & 77  & 80  & 79   \\
            \bottomrule
        \end{tabular}
        \caption{Number of problems solved}
    \end{subtable} \\

    \begin{subtable}[h]{\textwidth}
        \centering
        \begin{tabular}{lrrrrrrrrrr}
            \toprule
            \(N\) & 10   & 20   & 50   & 100  & 200  & 500  & 1000 \\
            \midrule
            std   & 1999 & 1864 & 1739 & 1604 & 1430 & 1127 & 1029 \\
            vf    & 2022 & 1870 & 1730 & 1644 & 1520 & 1231 & 1155 \\
            pastd & 2056 & 1863 & 1724 & 1662 & 1498 & 1315 & 1220 \\
            pvf   & 1980 & 1867 & 1770 & 1699 & 1486 & 1409 & 1429 \\
            cs1   & 2396 & 2915 & 3495 & 3585 & 3601 & 3601 & 3601 \\
            cs2   & 2069 & 1973 & 1787 & 1822 & 1697 & 1599 & 1518 \\
            pacs1 & 2566 & 3368 & 3601 & 3601 & 3601 & 3601 & 3601 \\
            pacs2 & 2135 & 2189 & 2438 & 2687 & 2890 & 2986 & 3050 \\
            vfcs1 & 2359 & 2818 & 3491 & 3587 & 3601 & 3601 & 3601 \\
            vfcs2 & 1992 & 1916 & 1852 & 1712 & 1673 & 1572 & 1676 \\
            pcs1  & 2351 & 2861 & 3448 & 3576 & 3601 & 3601 & 3601 \\
            pcs2  & 2012 & 1903 & 1696 & 1599 & 1544 & 1529 & 1607 \\
            \bottomrule
        \end{tabular}
        \caption{Time of the easy group}
    \end{subtable} \\

    \begin{subtable}[h]{\textwidth}
        \centering
        \begin{tabular}{lrrrrrrrrrr}
            \toprule
            \(N\) & 10   & 20   & 50   & 100  & 200  & 500   & 1000  \\
            \midrule
            std   & 12.8 & 12.7 & 11.9 & 11.3 & 10.5 & 9.3   & 8.5   \\
            vf    & 13.3 & 12.6 & 12.1 & 11.4 & 11.2 & 10.0  & 9.8   \\
            pastd & 13.2 & 12.8 & 11.7 & 11.7 & 10.8 & 9.7   & 9.2   \\
            pvf   & 13.6 & 12.7 & 12.1 & 11.7 & 11.1 & 10.1  & 10.6  \\
            cs1   & 16.3 & 21.8 & 42.1 & 70.0 & 94.2 & 99.6  & 100.0 \\
            cs2   & 13.3 & 12.9 & 12.8 & 12.2 & 11.9 & 10.9  & 10.9  \\
            pacs1 & 18.9 & 34.4 & 70.3 & 92.6 & 99.1 & 100.0 & 100.0 \\
            pacs2 & 13.7 & 13.4 & 14.6 & 15.1 & 15.5 & 17.1  & 19.2  \\
            vfcs1 & 16.3 & 20.3 & 36.3 & 58.2 & 86.2 & 97.7  & 99.5  \\
            vfcs2 & 13.6 & 13.5 & 13.5 & 13.3 & 12.8 & 12.1  & 12.1  \\
            pcs1  & 15.6 & 20.3 & 39.5 & 63.2 & 88.4 & 98.7  & 99.9  \\
            pcs2  & 13.2 & 13.3 & 12.4 & 11.8 & 10.8 & 10.6  & 11.1  \\
            \bottomrule
        \end{tabular}
        \caption{Gap of the hard group}
    \end{subtable}

    \caption{Performance of all formulations over \(N\).}
    \label{tab:models}
\end{table}

\begin{figure}
    \centering
    \input{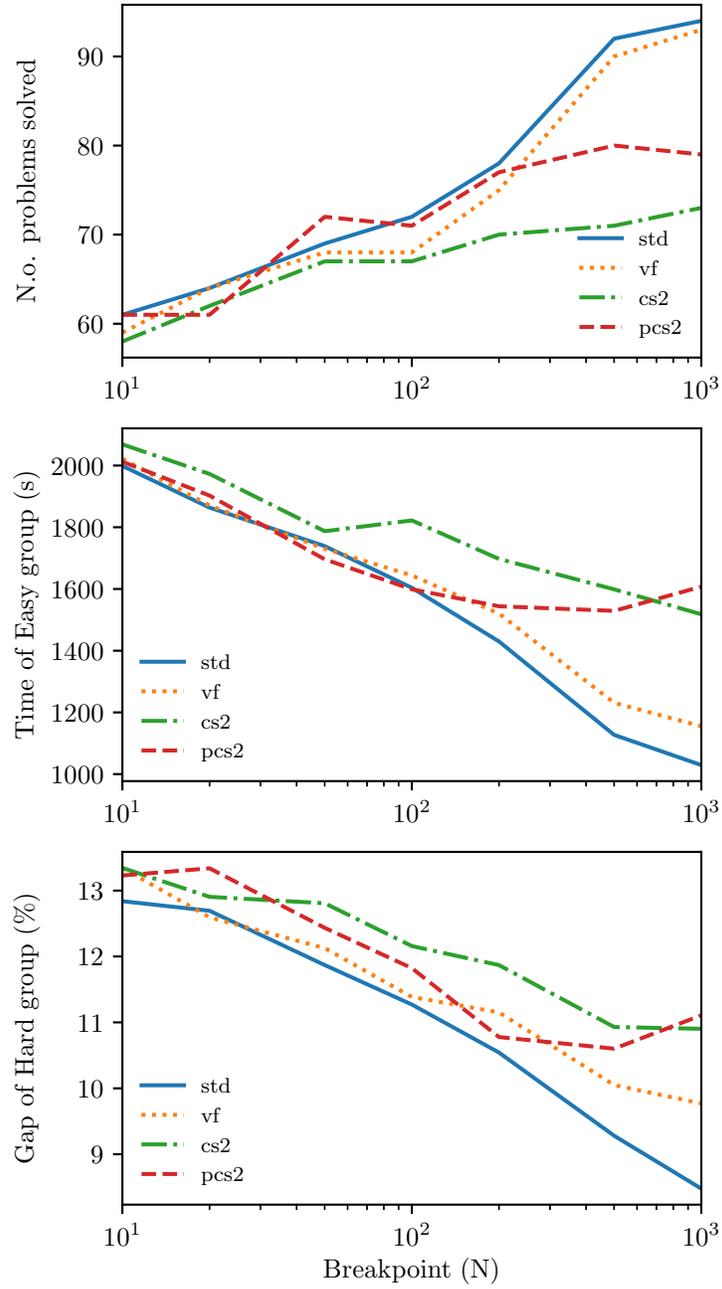}
    \caption{Performance of notable formulations over \(N\).}
    \label{fig:models}
\end{figure}

\subsection{Path-based Preprocessing}
In this section, we compare the new path-based preprocessing with the state-of-the-art preprocessing method (SPGM). Figure \ref{fig:reduceratio} shows the reduction ratio between those two
preprocessing methods. Given a breakpoint \(N\) (horizontal axis), the reduction ratio is the quotient of
the total number of nodes/arcs/tolled arcs of all commodities with no more than \(N\) bilevel feasible paths
after being processed by the path-based preprocessing over the same sum in the original graph (absolute ratio)
or in the SPGM-processed graph (relative ratio).
Each criterion (node, arc, tolled arc) has a different impact on the final single-level reformulation, which
is summarized in Table \ref{tab:criteria}. The tolled arc criterion is the most important because
it is linked with the number of binary variables in the final formulation. In Figure \ref{fig:reduceratio},
we can see that the path-based preprocessing excels in all three criteria. On average, graphs processed by
the path-based preprocessing have 10\% less nodes, 66\% less arcs, and 49\% less tolled arcs compare to the SPGM method.
Compared to the original graph, the path-based preprocessing removes 75\% of all tolled arcs, which reduces the size
of the problems significantly. The efficiency varies between different instance sets, with the path-based preprocessing being more favorable
in more connected graphs such as in dataset D and H.

\begin{figure}
    \centering
    \makebox[\textwidth][c]{\input{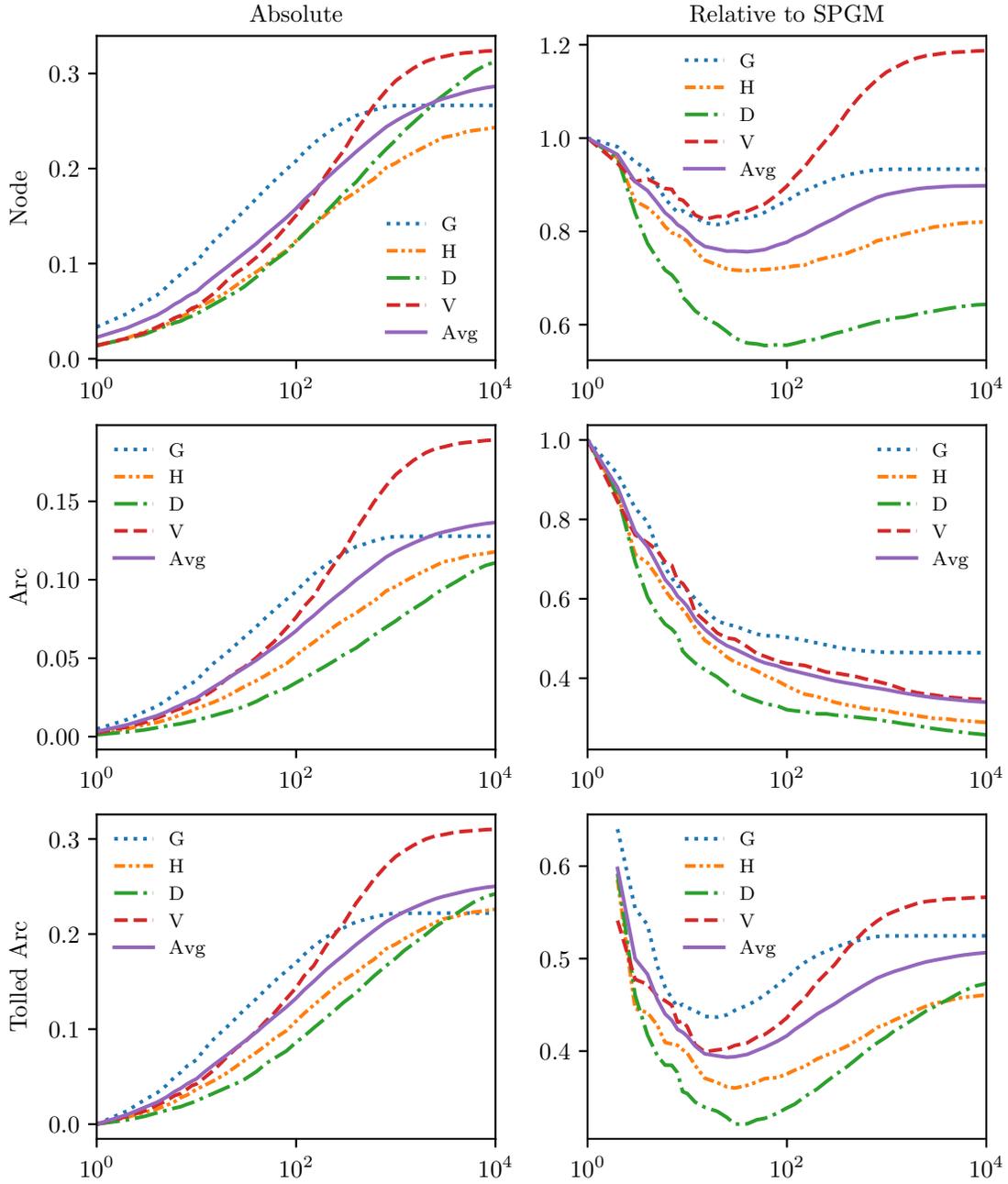}}
    \caption{Reduction ratio of Path-based Preprocessing.}
    \label{fig:reduceratio}
\end{figure}

\begin{table}
    \centering
    \begin{tabular}{ll}
        \toprule
        Criteria   & Impact                                                            \\
        \midrule
        Node       & \tabitem Num. of constraints in primal-arc                         \\
                   & \tabitem Num. of variables in dual-arc                             \\
        Arc        & \tabitem Num. of variables in primal-arc (strong duality)          \\
                   & \tabitem Num. of binary variables in primal-arc                    \\
                   & (complementary slackness)                                          \\
                   & \tabitem Num. of constraints in dual-arc                           \\
                   & \tabitem Num. of constraints in complementary slackness            \\
        Tolled arc & \tabitem Num. of binary variables in primal-arc                    \\
                   & \tabitem Num. of variables and constraints in direct linearization \\
        \bottomrule
    \end{tabular}
    \caption{Impact of the reduction ratio criteria.}
    \label{tab:criteria}
\end{table}

Figure \ref{fig:spgmperm} shows the real performance of the two preprocessing methods over \(N\).
The main formulation in both models is the (STD), but one is matched with path-based preprocessing,
while the other is matched with SPGM preprocessing.
Once again, we can observe that the path-based preprocessing excels in all three criteria. It can solve 30 more problems (over 200 in total),
in half the time and half the gap compared to the SPGM model. Besides that, the SPGM model hardly improves when \(N > 1000\)
while path-based preprocessing keeps thriving beyond this threshold.
In our experiments, applying SPGM on top of the path-based preprocessing does not produce significantly better results, and
sometimes can even be detrimental. This concludes that path-based preprocessing is the superior preprocessing method.

\begin{figure}
    \centering
    \input{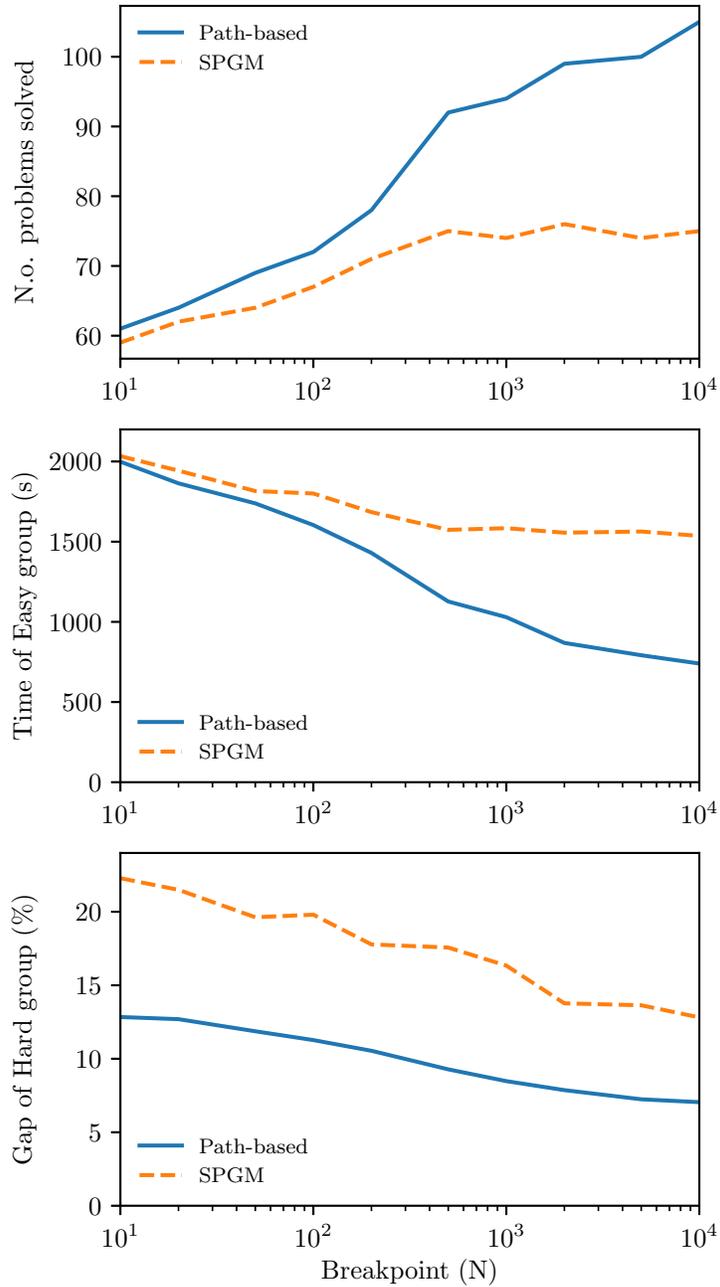}

    \caption{Performance comparison between Path-based preprocessing and SPGM preprocessing.}
    \label{fig:spgmperm}
\end{figure}

\subsection{Hybrid Model}
In this work, we show numerical results to justify the idea of hybrid model introduced in Section 4
as a way to compromise between the time spent for path enumeration and the time spent to solve the reformulation.
We will use the best reformulation which is the (STD) for all commodities in this test.
Figure \ref{fig:hybrid} shows its performance with \(N\) varying from 10 to 100000 (horizontal axis).
If a commodity has no more than \(N\) bilevel feasible paths, then path-based preprocessing is applied. Otherwise,
the unprocessed graph is used instead. The larger \(N\) is, the longer it takes to enumerate,
but more commodities will be processed. If \(N\) is very large, all commodities will be processed and all paths are enumerated.
If \(N\) is set to 0, then it is identical to an unprocessed model.
From Figure \ref{fig:hybrid}, we observe that larger \(N\) generally leads to better performance until a certain point around 10000 paths.
After that, the time spent for path enumeration does not have a positive impact on the performance anymore.
The average time of the easy group is not affected by large \(N\) because they do not have many paths
to enumerate in the first place.
At \(N = 100000\), there are two unsolved problems in class D which require more than 1 hour for path enumeration
and this contributes to the rise at the end in the average optimality gap of the hard group.
In conclusion, for the given instances, stopping the path enumeration at \(N = 10000\) provides a good balance.
Stopping early also improves the robustness of the program, because although commodities with more than 10000 paths are not common,
encountering one could make the program stuck for a long time.

\begin{figure}
    \centering
    \input{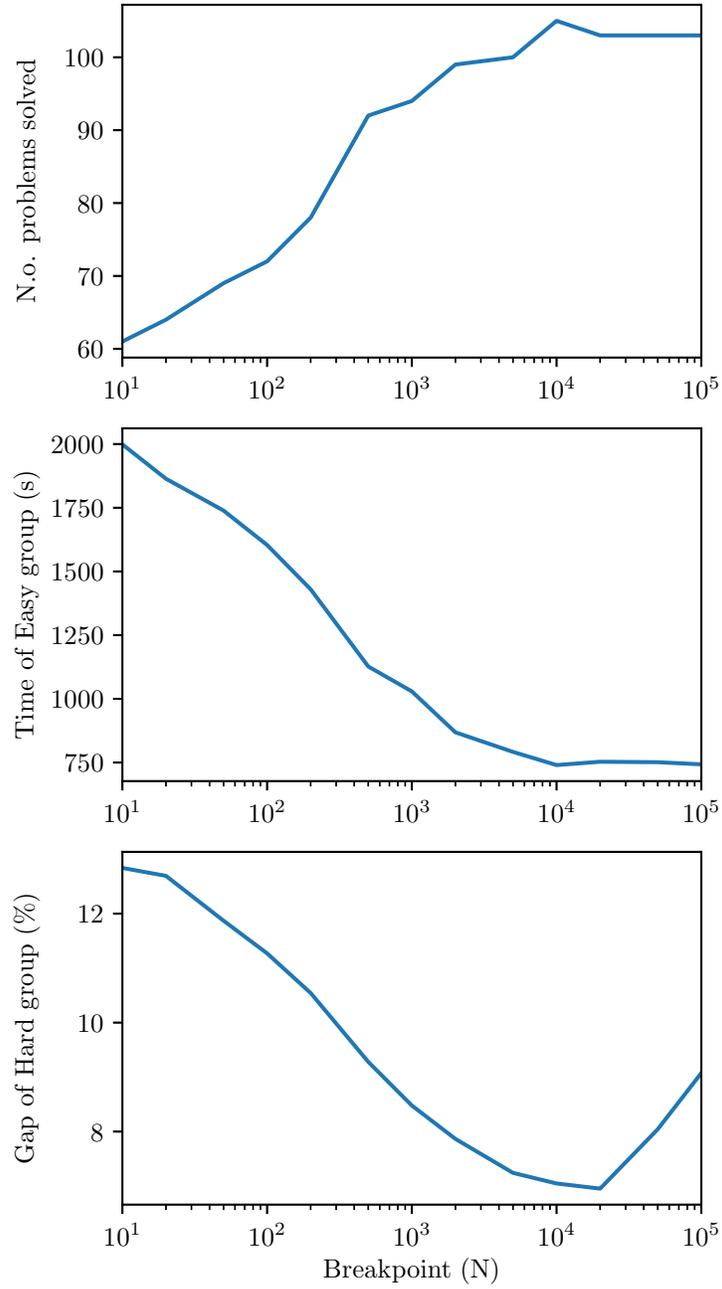}

    \caption{Performance of the standard model over \(N\).}
    \label{fig:hybrid}
\end{figure}

\section{Conclusion}
In this work, we related many modeling concepts which enrich and unify the general toolbox
to solve bilevel problems. If the follower problem can be formulated in multiple ways,
then we can mix the primal and the dual (if it exists) of any two formulations to write the
single-level formulation. In the case of a (mixed-integer) linear follower problem,
this alternative formulation always exists:
the follower problem can be written as a linear combination of
the extreme points of the convex hull of the feasible set.
Bilevel feasibility, and its enumeration, can be used as a preprocessing method. The hybrid framework is applicable
in any multi-commodity problem with multiple reformulations. The network pricing problem is a good example to demonstrate
all these modeling techniques together.

The NPP makes some theoretical assumptions whose relaxation enables a more close mirroring of reality. From a practical point of view, future work must consider additional real-life features, including arc construction cost, capacity on arcs and nodes, congestion, competition, and uncertainty,
many of which are highly non-linear and make the problem significantly harder.

\section*{Acknowledgements}

The authors are grateful for the support of Institut de valorisation des donn\'ees (IVADO) and Fonds de recherche du Qu\'ebec (FRQ) through the FRQ-IVADO Research Chair, and of the  Natural Sciences and Engineering Council of Canada (NSERC) through its  Discovery  Grant  program.

This research was enabled in part by support provided by  Calcul Qu\'ebec (\url{www.calculquebec.ca})
and Compute Canada (\url{www.computecanada.ca}).

\bibliographystyle{plain}
\bibliography{ref}

\end{document}

%% file: plots/numpaths.pgf
\begingroup%
\makeatletter%
\begin{pgfpicture}%
\pgfpathrectangle{\pgfpointorigin}{\pgfqpoint{4.000000in}{3.000000in}}%
\pgfusepath{use as bounding box, clip}%
\begin{pgfscope}%
\pgfsetbuttcap%
\pgfsetmiterjoin%
\definecolor{currentfill}{rgb}{1.000000,1.000000,1.000000}%
\pgfsetfillcolor{currentfill}%
\pgfsetlinewidth{0.000000pt}%
\definecolor{currentstroke}{rgb}{1.000000,1.000000,1.000000}%
\pgfsetstrokecolor{currentstroke}%
\pgfsetdash{}{0pt}%
\pgfpathmoveto{\pgfqpoint{0.000000in}{0.000000in}}%
\pgfpathlineto{\pgfqpoint{4.000000in}{0.000000in}}%
\pgfpathlineto{\pgfqpoint{4.000000in}{3.000000in}}%
\pgfpathlineto{\pgfqpoint{0.000000in}{3.000000in}}%
\pgfpathclose%
\pgfusepath{fill}%
\end{pgfscope}%
\begin{pgfscope}%
\pgfsetbuttcap%
\pgfsetmiterjoin%
\definecolor{currentfill}{rgb}{1.000000,1.000000,1.000000}%
\pgfsetfillcolor{currentfill}%
\pgfsetlinewidth{0.000000pt}%
\definecolor{currentstroke}{rgb}{0.000000,0.000000,0.000000}%
\pgfsetstrokecolor{currentstroke}%
\pgfsetstrokeopacity{0.000000}%
\pgfsetdash{}{0pt}%
\pgfpathmoveto{\pgfqpoint{0.603704in}{0.565123in}}%
\pgfpathlineto{\pgfqpoint{3.749402in}{0.565123in}}%
\pgfpathlineto{\pgfqpoint{3.749402in}{2.847069in}}%
\pgfpathlineto{\pgfqpoint{0.603704in}{2.847069in}}%
\pgfpathclose%
\pgfusepath{fill}%
\end{pgfscope}%
\begin{pgfscope}%
\pgfsetbuttcap%
\pgfsetroundjoin%
\definecolor{currentfill}{rgb}{0.000000,0.000000,0.000000}%
\pgfsetfillcolor{currentfill}%
\pgfsetlinewidth{0.803000pt}%
\definecolor{currentstroke}{rgb}{0.000000,0.000000,0.000000}%
\pgfsetstrokecolor{currentstroke}%
\pgfsetdash{}{0pt}%
\pgfsys@defobject{currentmarker}{\pgfqpoint{0.000000in}{-0.048611in}}{\pgfqpoint{0.000000in}{0.000000in}}{%
\pgfpathmoveto{\pgfqpoint{0.000000in}{0.000000in}}%
\pgfpathlineto{\pgfqpoint{0.000000in}{-0.048611in}}%
\pgfusepath{stroke,fill}%
}%
\begin{pgfscope}%
\pgfsys@transformshift{0.603704in}{0.565123in}%
\pgfsys@useobject{currentmarker}{}%
\end{pgfscope}%
\end{pgfscope}%
\begin{pgfscope}%
\definecolor{textcolor}{rgb}{0.000000,0.000000,0.000000}%
\pgfsetstrokecolor{textcolor}%
\pgfsetfillcolor{textcolor}%
\pgftext[x=0.603704in,y=0.467901in,,top]{\color{textcolor}\rmfamily\fontsize{10.000000}{12.000000}\selectfont \(\displaystyle {10^{0}}\)}%
\end{pgfscope}%
\begin{pgfscope}%
\pgfsetbuttcap%
\pgfsetroundjoin%
\definecolor{currentfill}{rgb}{0.000000,0.000000,0.000000}%
\pgfsetfillcolor{currentfill}%
\pgfsetlinewidth{0.803000pt}%
\definecolor{currentstroke}{rgb}{0.000000,0.000000,0.000000}%
\pgfsetstrokecolor{currentstroke}%
\pgfsetdash{}{0pt}%
\pgfsys@defobject{currentmarker}{\pgfqpoint{0.000000in}{-0.048611in}}{\pgfqpoint{0.000000in}{0.000000in}}{%
\pgfpathmoveto{\pgfqpoint{0.000000in}{0.000000in}}%
\pgfpathlineto{\pgfqpoint{0.000000in}{-0.048611in}}%
\pgfusepath{stroke,fill}%
}%
\begin{pgfscope}%
\pgfsys@transformshift{1.390129in}{0.565123in}%
\pgfsys@useobject{currentmarker}{}%
\end{pgfscope}%
\end{pgfscope}%
\begin{pgfscope}%
\definecolor{textcolor}{rgb}{0.000000,0.000000,0.000000}%
\pgfsetstrokecolor{textcolor}%
\pgfsetfillcolor{textcolor}%
\pgftext[x=1.390129in,y=0.467901in,,top]{\color{textcolor}\rmfamily\fontsize{10.000000}{12.000000}\selectfont \(\displaystyle {10^{1}}\)}%
\end{pgfscope}%
\begin{pgfscope}%
\pgfsetbuttcap%
\pgfsetroundjoin%
\definecolor{currentfill}{rgb}{0.000000,0.000000,0.000000}%
\pgfsetfillcolor{currentfill}%
\pgfsetlinewidth{0.803000pt}%
\definecolor{currentstroke}{rgb}{0.000000,0.000000,0.000000}%
\pgfsetstrokecolor{currentstroke}%
\pgfsetdash{}{0pt}%
\pgfsys@defobject{currentmarker}{\pgfqpoint{0.000000in}{-0.048611in}}{\pgfqpoint{0.000000in}{0.000000in}}{%
\pgfpathmoveto{\pgfqpoint{0.000000in}{0.000000in}}%
\pgfpathlineto{\pgfqpoint{0.000000in}{-0.048611in}}%
\pgfusepath{stroke,fill}%
}%
\begin{pgfscope}%
\pgfsys@transformshift{2.176553in}{0.565123in}%
\pgfsys@useobject{currentmarker}{}%
\end{pgfscope}%
\end{pgfscope}%
\begin{pgfscope}%
\definecolor{textcolor}{rgb}{0.000000,0.000000,0.000000}%
\pgfsetstrokecolor{textcolor}%
\pgfsetfillcolor{textcolor}%
\pgftext[x=2.176553in,y=0.467901in,,top]{\color{textcolor}\rmfamily\fontsize{10.000000}{12.000000}\selectfont \(\displaystyle {10^{2}}\)}%
\end{pgfscope}%
\begin{pgfscope}%
\pgfsetbuttcap%
\pgfsetroundjoin%
\definecolor{currentfill}{rgb}{0.000000,0.000000,0.000000}%
\pgfsetfillcolor{currentfill}%
\pgfsetlinewidth{0.803000pt}%
\definecolor{currentstroke}{rgb}{0.000000,0.000000,0.000000}%
\pgfsetstrokecolor{currentstroke}%
\pgfsetdash{}{0pt}%
\pgfsys@defobject{currentmarker}{\pgfqpoint{0.000000in}{-0.048611in}}{\pgfqpoint{0.000000in}{0.000000in}}{%
\pgfpathmoveto{\pgfqpoint{0.000000in}{0.000000in}}%
\pgfpathlineto{\pgfqpoint{0.000000in}{-0.048611in}}%
\pgfusepath{stroke,fill}%
}%
\begin{pgfscope}%
\pgfsys@transformshift{2.962977in}{0.565123in}%
\pgfsys@useobject{currentmarker}{}%
\end{pgfscope}%
\end{pgfscope}%
\begin{pgfscope}%
\definecolor{textcolor}{rgb}{0.000000,0.000000,0.000000}%
\pgfsetstrokecolor{textcolor}%
\pgfsetfillcolor{textcolor}%
\pgftext[x=2.962977in,y=0.467901in,,top]{\color{textcolor}\rmfamily\fontsize{10.000000}{12.000000}\selectfont \(\displaystyle {10^{3}}\)}%
\end{pgfscope}%
\begin{pgfscope}%
\pgfsetbuttcap%
\pgfsetroundjoin%
\definecolor{currentfill}{rgb}{0.000000,0.000000,0.000000}%
\pgfsetfillcolor{currentfill}%
\pgfsetlinewidth{0.803000pt}%
\definecolor{currentstroke}{rgb}{0.000000,0.000000,0.000000}%
\pgfsetstrokecolor{currentstroke}%
\pgfsetdash{}{0pt}%
\pgfsys@defobject{currentmarker}{\pgfqpoint{0.000000in}{-0.048611in}}{\pgfqpoint{0.000000in}{0.000000in}}{%
\pgfpathmoveto{\pgfqpoint{0.000000in}{0.000000in}}%
\pgfpathlineto{\pgfqpoint{0.000000in}{-0.048611in}}%
\pgfusepath{stroke,fill}%
}%
\begin{pgfscope}%
\pgfsys@transformshift{3.749402in}{0.565123in}%
\pgfsys@useobject{currentmarker}{}%
\end{pgfscope}%
\end{pgfscope}%
\begin{pgfscope}%
\definecolor{textcolor}{rgb}{0.000000,0.000000,0.000000}%
\pgfsetstrokecolor{textcolor}%
\pgfsetfillcolor{textcolor}%
\pgftext[x=3.749402in,y=0.467901in,,top]{\color{textcolor}\rmfamily\fontsize{10.000000}{12.000000}\selectfont \(\displaystyle {10^{4}}\)}%
\end{pgfscope}%
\begin{pgfscope}%
\pgfsetbuttcap%
\pgfsetroundjoin%
\definecolor{currentfill}{rgb}{0.000000,0.000000,0.000000}%
\pgfsetfillcolor{currentfill}%
\pgfsetlinewidth{0.602250pt}%
\definecolor{currentstroke}{rgb}{0.000000,0.000000,0.000000}%
\pgfsetstrokecolor{currentstroke}%
\pgfsetdash{}{0pt}%
\pgfsys@defobject{currentmarker}{\pgfqpoint{0.000000in}{-0.027778in}}{\pgfqpoint{0.000000in}{0.000000in}}{%
\pgfpathmoveto{\pgfqpoint{0.000000in}{0.000000in}}%
\pgfpathlineto{\pgfqpoint{0.000000in}{-0.027778in}}%
\pgfusepath{stroke,fill}%
}%
\begin{pgfscope}%
\pgfsys@transformshift{0.840441in}{0.565123in}%
\pgfsys@useobject{currentmarker}{}%
\end{pgfscope}%
\end{pgfscope}%
\begin{pgfscope}%
\pgfsetbuttcap%
\pgfsetroundjoin%
\definecolor{currentfill}{rgb}{0.000000,0.000000,0.000000}%
\pgfsetfillcolor{currentfill}%
\pgfsetlinewidth{0.602250pt}%
\definecolor{currentstroke}{rgb}{0.000000,0.000000,0.000000}%
\pgfsetstrokecolor{currentstroke}%
\pgfsetdash{}{0pt}%
\pgfsys@defobject{currentmarker}{\pgfqpoint{0.000000in}{-0.027778in}}{\pgfqpoint{0.000000in}{0.000000in}}{%
\pgfpathmoveto{\pgfqpoint{0.000000in}{0.000000in}}%
\pgfpathlineto{\pgfqpoint{0.000000in}{-0.027778in}}%
\pgfusepath{stroke,fill}%
}%
\begin{pgfscope}%
\pgfsys@transformshift{0.978924in}{0.565123in}%
\pgfsys@useobject{currentmarker}{}%
\end{pgfscope}%
\end{pgfscope}%
\begin{pgfscope}%
\pgfsetbuttcap%
\pgfsetroundjoin%
\definecolor{currentfill}{rgb}{0.000000,0.000000,0.000000}%
\pgfsetfillcolor{currentfill}%
\pgfsetlinewidth{0.602250pt}%
\definecolor{currentstroke}{rgb}{0.000000,0.000000,0.000000}%
\pgfsetstrokecolor{currentstroke}%
\pgfsetdash{}{0pt}%
\pgfsys@defobject{currentmarker}{\pgfqpoint{0.000000in}{-0.027778in}}{\pgfqpoint{0.000000in}{0.000000in}}{%
\pgfpathmoveto{\pgfqpoint{0.000000in}{0.000000in}}%
\pgfpathlineto{\pgfqpoint{0.000000in}{-0.027778in}}%
\pgfusepath{stroke,fill}%
}%
\begin{pgfscope}%
\pgfsys@transformshift{1.077179in}{0.565123in}%
\pgfsys@useobject{currentmarker}{}%
\end{pgfscope}%
\end{pgfscope}%
\begin{pgfscope}%
\pgfsetbuttcap%
\pgfsetroundjoin%
\definecolor{currentfill}{rgb}{0.000000,0.000000,0.000000}%
\pgfsetfillcolor{currentfill}%
\pgfsetlinewidth{0.602250pt}%
\definecolor{currentstroke}{rgb}{0.000000,0.000000,0.000000}%
\pgfsetstrokecolor{currentstroke}%
\pgfsetdash{}{0pt}%
\pgfsys@defobject{currentmarker}{\pgfqpoint{0.000000in}{-0.027778in}}{\pgfqpoint{0.000000in}{0.000000in}}{%
\pgfpathmoveto{\pgfqpoint{0.000000in}{0.000000in}}%
\pgfpathlineto{\pgfqpoint{0.000000in}{-0.027778in}}%
\pgfusepath{stroke,fill}%
}%
\begin{pgfscope}%
\pgfsys@transformshift{1.153391in}{0.565123in}%
\pgfsys@useobject{currentmarker}{}%
\end{pgfscope}%
\end{pgfscope}%
\begin{pgfscope}%
\pgfsetbuttcap%
\pgfsetroundjoin%
\definecolor{currentfill}{rgb}{0.000000,0.000000,0.000000}%
\pgfsetfillcolor{currentfill}%
\pgfsetlinewidth{0.602250pt}%
\definecolor{currentstroke}{rgb}{0.000000,0.000000,0.000000}%
\pgfsetstrokecolor{currentstroke}%
\pgfsetdash{}{0pt}%
\pgfsys@defobject{currentmarker}{\pgfqpoint{0.000000in}{-0.027778in}}{\pgfqpoint{0.000000in}{0.000000in}}{%
\pgfpathmoveto{\pgfqpoint{0.000000in}{0.000000in}}%
\pgfpathlineto{\pgfqpoint{0.000000in}{-0.027778in}}%
\pgfusepath{stroke,fill}%
}%
\begin{pgfscope}%
\pgfsys@transformshift{1.215661in}{0.565123in}%
\pgfsys@useobject{currentmarker}{}%
\end{pgfscope}%
\end{pgfscope}%
\begin{pgfscope}%
\pgfsetbuttcap%
\pgfsetroundjoin%
\definecolor{currentfill}{rgb}{0.000000,0.000000,0.000000}%
\pgfsetfillcolor{currentfill}%
\pgfsetlinewidth{0.602250pt}%
\definecolor{currentstroke}{rgb}{0.000000,0.000000,0.000000}%
\pgfsetstrokecolor{currentstroke}%
\pgfsetdash{}{0pt}%
\pgfsys@defobject{currentmarker}{\pgfqpoint{0.000000in}{-0.027778in}}{\pgfqpoint{0.000000in}{0.000000in}}{%
\pgfpathmoveto{\pgfqpoint{0.000000in}{0.000000in}}%
\pgfpathlineto{\pgfqpoint{0.000000in}{-0.027778in}}%
\pgfusepath{stroke,fill}%
}%
\begin{pgfscope}%
\pgfsys@transformshift{1.268310in}{0.565123in}%
\pgfsys@useobject{currentmarker}{}%
\end{pgfscope}%
\end{pgfscope}%
\begin{pgfscope}%
\pgfsetbuttcap%
\pgfsetroundjoin%
\definecolor{currentfill}{rgb}{0.000000,0.000000,0.000000}%
\pgfsetfillcolor{currentfill}%
\pgfsetlinewidth{0.602250pt}%
\definecolor{currentstroke}{rgb}{0.000000,0.000000,0.000000}%
\pgfsetstrokecolor{currentstroke}%
\pgfsetdash{}{0pt}%
\pgfsys@defobject{currentmarker}{\pgfqpoint{0.000000in}{-0.027778in}}{\pgfqpoint{0.000000in}{0.000000in}}{%
\pgfpathmoveto{\pgfqpoint{0.000000in}{0.000000in}}%
\pgfpathlineto{\pgfqpoint{0.000000in}{-0.027778in}}%
\pgfusepath{stroke,fill}%
}%
\begin{pgfscope}%
\pgfsys@transformshift{1.313916in}{0.565123in}%
\pgfsys@useobject{currentmarker}{}%
\end{pgfscope}%
\end{pgfscope}%
\begin{pgfscope}%
\pgfsetbuttcap%
\pgfsetroundjoin%
\definecolor{currentfill}{rgb}{0.000000,0.000000,0.000000}%
\pgfsetfillcolor{currentfill}%
\pgfsetlinewidth{0.602250pt}%
\definecolor{currentstroke}{rgb}{0.000000,0.000000,0.000000}%
\pgfsetstrokecolor{currentstroke}%
\pgfsetdash{}{0pt}%
\pgfsys@defobject{currentmarker}{\pgfqpoint{0.000000in}{-0.027778in}}{\pgfqpoint{0.000000in}{0.000000in}}{%
\pgfpathmoveto{\pgfqpoint{0.000000in}{0.000000in}}%
\pgfpathlineto{\pgfqpoint{0.000000in}{-0.027778in}}%
\pgfusepath{stroke,fill}%
}%
\begin{pgfscope}%
\pgfsys@transformshift{1.354144in}{0.565123in}%
\pgfsys@useobject{currentmarker}{}%
\end{pgfscope}%
\end{pgfscope}%
\begin{pgfscope}%
\pgfsetbuttcap%
\pgfsetroundjoin%
\definecolor{currentfill}{rgb}{0.000000,0.000000,0.000000}%
\pgfsetfillcolor{currentfill}%
\pgfsetlinewidth{0.602250pt}%
\definecolor{currentstroke}{rgb}{0.000000,0.000000,0.000000}%
\pgfsetstrokecolor{currentstroke}%
\pgfsetdash{}{0pt}%
\pgfsys@defobject{currentmarker}{\pgfqpoint{0.000000in}{-0.027778in}}{\pgfqpoint{0.000000in}{0.000000in}}{%
\pgfpathmoveto{\pgfqpoint{0.000000in}{0.000000in}}%
\pgfpathlineto{\pgfqpoint{0.000000in}{-0.027778in}}%
\pgfusepath{stroke,fill}%
}%
\begin{pgfscope}%
\pgfsys@transformshift{1.626866in}{0.565123in}%
\pgfsys@useobject{currentmarker}{}%
\end{pgfscope}%
\end{pgfscope}%
\begin{pgfscope}%
\pgfsetbuttcap%
\pgfsetroundjoin%
\definecolor{currentfill}{rgb}{0.000000,0.000000,0.000000}%
\pgfsetfillcolor{currentfill}%
\pgfsetlinewidth{0.602250pt}%
\definecolor{currentstroke}{rgb}{0.000000,0.000000,0.000000}%
\pgfsetstrokecolor{currentstroke}%
\pgfsetdash{}{0pt}%
\pgfsys@defobject{currentmarker}{\pgfqpoint{0.000000in}{-0.027778in}}{\pgfqpoint{0.000000in}{0.000000in}}{%
\pgfpathmoveto{\pgfqpoint{0.000000in}{0.000000in}}%
\pgfpathlineto{\pgfqpoint{0.000000in}{-0.027778in}}%
\pgfusepath{stroke,fill}%
}%
\begin{pgfscope}%
\pgfsys@transformshift{1.765348in}{0.565123in}%
\pgfsys@useobject{currentmarker}{}%
\end{pgfscope}%
\end{pgfscope}%
\begin{pgfscope}%
\pgfsetbuttcap%
\pgfsetroundjoin%
\definecolor{currentfill}{rgb}{0.000000,0.000000,0.000000}%
\pgfsetfillcolor{currentfill}%
\pgfsetlinewidth{0.602250pt}%
\definecolor{currentstroke}{rgb}{0.000000,0.000000,0.000000}%
\pgfsetstrokecolor{currentstroke}%
\pgfsetdash{}{0pt}%
\pgfsys@defobject{currentmarker}{\pgfqpoint{0.000000in}{-0.027778in}}{\pgfqpoint{0.000000in}{0.000000in}}{%
\pgfpathmoveto{\pgfqpoint{0.000000in}{0.000000in}}%
\pgfpathlineto{\pgfqpoint{0.000000in}{-0.027778in}}%
\pgfusepath{stroke,fill}%
}%
\begin{pgfscope}%
\pgfsys@transformshift{1.863603in}{0.565123in}%
\pgfsys@useobject{currentmarker}{}%
\end{pgfscope}%
\end{pgfscope}%
\begin{pgfscope}%
\pgfsetbuttcap%
\pgfsetroundjoin%
\definecolor{currentfill}{rgb}{0.000000,0.000000,0.000000}%
\pgfsetfillcolor{currentfill}%
\pgfsetlinewidth{0.602250pt}%
\definecolor{currentstroke}{rgb}{0.000000,0.000000,0.000000}%
\pgfsetstrokecolor{currentstroke}%
\pgfsetdash{}{0pt}%
\pgfsys@defobject{currentmarker}{\pgfqpoint{0.000000in}{-0.027778in}}{\pgfqpoint{0.000000in}{0.000000in}}{%
\pgfpathmoveto{\pgfqpoint{0.000000in}{0.000000in}}%
\pgfpathlineto{\pgfqpoint{0.000000in}{-0.027778in}}%
\pgfusepath{stroke,fill}%
}%
\begin{pgfscope}%
\pgfsys@transformshift{1.939816in}{0.565123in}%
\pgfsys@useobject{currentmarker}{}%
\end{pgfscope}%
\end{pgfscope}%
\begin{pgfscope}%
\pgfsetbuttcap%
\pgfsetroundjoin%
\definecolor{currentfill}{rgb}{0.000000,0.000000,0.000000}%
\pgfsetfillcolor{currentfill}%
\pgfsetlinewidth{0.602250pt}%
\definecolor{currentstroke}{rgb}{0.000000,0.000000,0.000000}%
\pgfsetstrokecolor{currentstroke}%
\pgfsetdash{}{0pt}%
\pgfsys@defobject{currentmarker}{\pgfqpoint{0.000000in}{-0.027778in}}{\pgfqpoint{0.000000in}{0.000000in}}{%
\pgfpathmoveto{\pgfqpoint{0.000000in}{0.000000in}}%
\pgfpathlineto{\pgfqpoint{0.000000in}{-0.027778in}}%
\pgfusepath{stroke,fill}%
}%
\begin{pgfscope}%
\pgfsys@transformshift{2.002086in}{0.565123in}%
\pgfsys@useobject{currentmarker}{}%
\end{pgfscope}%
\end{pgfscope}%
\begin{pgfscope}%
\pgfsetbuttcap%
\pgfsetroundjoin%
\definecolor{currentfill}{rgb}{0.000000,0.000000,0.000000}%
\pgfsetfillcolor{currentfill}%
\pgfsetlinewidth{0.602250pt}%
\definecolor{currentstroke}{rgb}{0.000000,0.000000,0.000000}%
\pgfsetstrokecolor{currentstroke}%
\pgfsetdash{}{0pt}%
\pgfsys@defobject{currentmarker}{\pgfqpoint{0.000000in}{-0.027778in}}{\pgfqpoint{0.000000in}{0.000000in}}{%
\pgfpathmoveto{\pgfqpoint{0.000000in}{0.000000in}}%
\pgfpathlineto{\pgfqpoint{0.000000in}{-0.027778in}}%
\pgfusepath{stroke,fill}%
}%
\begin{pgfscope}%
\pgfsys@transformshift{2.054734in}{0.565123in}%
\pgfsys@useobject{currentmarker}{}%
\end{pgfscope}%
\end{pgfscope}%
\begin{pgfscope}%
\pgfsetbuttcap%
\pgfsetroundjoin%
\definecolor{currentfill}{rgb}{0.000000,0.000000,0.000000}%
\pgfsetfillcolor{currentfill}%
\pgfsetlinewidth{0.602250pt}%
\definecolor{currentstroke}{rgb}{0.000000,0.000000,0.000000}%
\pgfsetstrokecolor{currentstroke}%
\pgfsetdash{}{0pt}%
\pgfsys@defobject{currentmarker}{\pgfqpoint{0.000000in}{-0.027778in}}{\pgfqpoint{0.000000in}{0.000000in}}{%
\pgfpathmoveto{\pgfqpoint{0.000000in}{0.000000in}}%
\pgfpathlineto{\pgfqpoint{0.000000in}{-0.027778in}}%
\pgfusepath{stroke,fill}%
}%
\begin{pgfscope}%
\pgfsys@transformshift{2.100341in}{0.565123in}%
\pgfsys@useobject{currentmarker}{}%
\end{pgfscope}%
\end{pgfscope}%
\begin{pgfscope}%
\pgfsetbuttcap%
\pgfsetroundjoin%
\definecolor{currentfill}{rgb}{0.000000,0.000000,0.000000}%
\pgfsetfillcolor{currentfill}%
\pgfsetlinewidth{0.602250pt}%
\definecolor{currentstroke}{rgb}{0.000000,0.000000,0.000000}%
\pgfsetstrokecolor{currentstroke}%
\pgfsetdash{}{0pt}%
\pgfsys@defobject{currentmarker}{\pgfqpoint{0.000000in}{-0.027778in}}{\pgfqpoint{0.000000in}{0.000000in}}{%
\pgfpathmoveto{\pgfqpoint{0.000000in}{0.000000in}}%
\pgfpathlineto{\pgfqpoint{0.000000in}{-0.027778in}}%
\pgfusepath{stroke,fill}%
}%
\begin{pgfscope}%
\pgfsys@transformshift{2.140568in}{0.565123in}%
\pgfsys@useobject{currentmarker}{}%
\end{pgfscope}%
\end{pgfscope}%
\begin{pgfscope}%
\pgfsetbuttcap%
\pgfsetroundjoin%
\definecolor{currentfill}{rgb}{0.000000,0.000000,0.000000}%
\pgfsetfillcolor{currentfill}%
\pgfsetlinewidth{0.602250pt}%
\definecolor{currentstroke}{rgb}{0.000000,0.000000,0.000000}%
\pgfsetstrokecolor{currentstroke}%
\pgfsetdash{}{0pt}%
\pgfsys@defobject{currentmarker}{\pgfqpoint{0.000000in}{-0.027778in}}{\pgfqpoint{0.000000in}{0.000000in}}{%
\pgfpathmoveto{\pgfqpoint{0.000000in}{0.000000in}}%
\pgfpathlineto{\pgfqpoint{0.000000in}{-0.027778in}}%
\pgfusepath{stroke,fill}%
}%
\begin{pgfscope}%
\pgfsys@transformshift{2.413290in}{0.565123in}%
\pgfsys@useobject{currentmarker}{}%
\end{pgfscope}%
\end{pgfscope}%
\begin{pgfscope}%
\pgfsetbuttcap%
\pgfsetroundjoin%
\definecolor{currentfill}{rgb}{0.000000,0.000000,0.000000}%
\pgfsetfillcolor{currentfill}%
\pgfsetlinewidth{0.602250pt}%
\definecolor{currentstroke}{rgb}{0.000000,0.000000,0.000000}%
\pgfsetstrokecolor{currentstroke}%
\pgfsetdash{}{0pt}%
\pgfsys@defobject{currentmarker}{\pgfqpoint{0.000000in}{-0.027778in}}{\pgfqpoint{0.000000in}{0.000000in}}{%
\pgfpathmoveto{\pgfqpoint{0.000000in}{0.000000in}}%
\pgfpathlineto{\pgfqpoint{0.000000in}{-0.027778in}}%
\pgfusepath{stroke,fill}%
}%
\begin{pgfscope}%
\pgfsys@transformshift{2.551773in}{0.565123in}%
\pgfsys@useobject{currentmarker}{}%
\end{pgfscope}%
\end{pgfscope}%
\begin{pgfscope}%
\pgfsetbuttcap%
\pgfsetroundjoin%
\definecolor{currentfill}{rgb}{0.000000,0.000000,0.000000}%
\pgfsetfillcolor{currentfill}%
\pgfsetlinewidth{0.602250pt}%
\definecolor{currentstroke}{rgb}{0.000000,0.000000,0.000000}%
\pgfsetstrokecolor{currentstroke}%
\pgfsetdash{}{0pt}%
\pgfsys@defobject{currentmarker}{\pgfqpoint{0.000000in}{-0.027778in}}{\pgfqpoint{0.000000in}{0.000000in}}{%
\pgfpathmoveto{\pgfqpoint{0.000000in}{0.000000in}}%
\pgfpathlineto{\pgfqpoint{0.000000in}{-0.027778in}}%
\pgfusepath{stroke,fill}%
}%
\begin{pgfscope}%
\pgfsys@transformshift{2.650028in}{0.565123in}%
\pgfsys@useobject{currentmarker}{}%
\end{pgfscope}%
\end{pgfscope}%
\begin{pgfscope}%
\pgfsetbuttcap%
\pgfsetroundjoin%
\definecolor{currentfill}{rgb}{0.000000,0.000000,0.000000}%
\pgfsetfillcolor{currentfill}%
\pgfsetlinewidth{0.602250pt}%
\definecolor{currentstroke}{rgb}{0.000000,0.000000,0.000000}%
\pgfsetstrokecolor{currentstroke}%
\pgfsetdash{}{0pt}%
\pgfsys@defobject{currentmarker}{\pgfqpoint{0.000000in}{-0.027778in}}{\pgfqpoint{0.000000in}{0.000000in}}{%
\pgfpathmoveto{\pgfqpoint{0.000000in}{0.000000in}}%
\pgfpathlineto{\pgfqpoint{0.000000in}{-0.027778in}}%
\pgfusepath{stroke,fill}%
}%
\begin{pgfscope}%
\pgfsys@transformshift{2.726240in}{0.565123in}%
\pgfsys@useobject{currentmarker}{}%
\end{pgfscope}%
\end{pgfscope}%
\begin{pgfscope}%
\pgfsetbuttcap%
\pgfsetroundjoin%
\definecolor{currentfill}{rgb}{0.000000,0.000000,0.000000}%
\pgfsetfillcolor{currentfill}%
\pgfsetlinewidth{0.602250pt}%
\definecolor{currentstroke}{rgb}{0.000000,0.000000,0.000000}%
\pgfsetstrokecolor{currentstroke}%
\pgfsetdash{}{0pt}%
\pgfsys@defobject{currentmarker}{\pgfqpoint{0.000000in}{-0.027778in}}{\pgfqpoint{0.000000in}{0.000000in}}{%
\pgfpathmoveto{\pgfqpoint{0.000000in}{0.000000in}}%
\pgfpathlineto{\pgfqpoint{0.000000in}{-0.027778in}}%
\pgfusepath{stroke,fill}%
}%
\begin{pgfscope}%
\pgfsys@transformshift{2.788510in}{0.565123in}%
\pgfsys@useobject{currentmarker}{}%
\end{pgfscope}%
\end{pgfscope}%
\begin{pgfscope}%
\pgfsetbuttcap%
\pgfsetroundjoin%
\definecolor{currentfill}{rgb}{0.000000,0.000000,0.000000}%
\pgfsetfillcolor{currentfill}%
\pgfsetlinewidth{0.602250pt}%
\definecolor{currentstroke}{rgb}{0.000000,0.000000,0.000000}%
\pgfsetstrokecolor{currentstroke}%
\pgfsetdash{}{0pt}%
\pgfsys@defobject{currentmarker}{\pgfqpoint{0.000000in}{-0.027778in}}{\pgfqpoint{0.000000in}{0.000000in}}{%
\pgfpathmoveto{\pgfqpoint{0.000000in}{0.000000in}}%
\pgfpathlineto{\pgfqpoint{0.000000in}{-0.027778in}}%
\pgfusepath{stroke,fill}%
}%
\begin{pgfscope}%
\pgfsys@transformshift{2.841159in}{0.565123in}%
\pgfsys@useobject{currentmarker}{}%
\end{pgfscope}%
\end{pgfscope}%
\begin{pgfscope}%
\pgfsetbuttcap%
\pgfsetroundjoin%
\definecolor{currentfill}{rgb}{0.000000,0.000000,0.000000}%
\pgfsetfillcolor{currentfill}%
\pgfsetlinewidth{0.602250pt}%
\definecolor{currentstroke}{rgb}{0.000000,0.000000,0.000000}%
\pgfsetstrokecolor{currentstroke}%
\pgfsetdash{}{0pt}%
\pgfsys@defobject{currentmarker}{\pgfqpoint{0.000000in}{-0.027778in}}{\pgfqpoint{0.000000in}{0.000000in}}{%
\pgfpathmoveto{\pgfqpoint{0.000000in}{0.000000in}}%
\pgfpathlineto{\pgfqpoint{0.000000in}{-0.027778in}}%
\pgfusepath{stroke,fill}%
}%
\begin{pgfscope}%
\pgfsys@transformshift{2.886765in}{0.565123in}%
\pgfsys@useobject{currentmarker}{}%
\end{pgfscope}%
\end{pgfscope}%
\begin{pgfscope}%
\pgfsetbuttcap%
\pgfsetroundjoin%
\definecolor{currentfill}{rgb}{0.000000,0.000000,0.000000}%
\pgfsetfillcolor{currentfill}%
\pgfsetlinewidth{0.602250pt}%
\definecolor{currentstroke}{rgb}{0.000000,0.000000,0.000000}%
\pgfsetstrokecolor{currentstroke}%
\pgfsetdash{}{0pt}%
\pgfsys@defobject{currentmarker}{\pgfqpoint{0.000000in}{-0.027778in}}{\pgfqpoint{0.000000in}{0.000000in}}{%
\pgfpathmoveto{\pgfqpoint{0.000000in}{0.000000in}}%
\pgfpathlineto{\pgfqpoint{0.000000in}{-0.027778in}}%
\pgfusepath{stroke,fill}%
}%
\begin{pgfscope}%
\pgfsys@transformshift{2.926992in}{0.565123in}%
\pgfsys@useobject{currentmarker}{}%
\end{pgfscope}%
\end{pgfscope}%
\begin{pgfscope}%
\pgfsetbuttcap%
\pgfsetroundjoin%
\definecolor{currentfill}{rgb}{0.000000,0.000000,0.000000}%
\pgfsetfillcolor{currentfill}%
\pgfsetlinewidth{0.602250pt}%
\definecolor{currentstroke}{rgb}{0.000000,0.000000,0.000000}%
\pgfsetstrokecolor{currentstroke}%
\pgfsetdash{}{0pt}%
\pgfsys@defobject{currentmarker}{\pgfqpoint{0.000000in}{-0.027778in}}{\pgfqpoint{0.000000in}{0.000000in}}{%
\pgfpathmoveto{\pgfqpoint{0.000000in}{0.000000in}}%
\pgfpathlineto{\pgfqpoint{0.000000in}{-0.027778in}}%
\pgfusepath{stroke,fill}%
}%
\begin{pgfscope}%
\pgfsys@transformshift{3.199715in}{0.565123in}%
\pgfsys@useobject{currentmarker}{}%
\end{pgfscope}%
\end{pgfscope}%
\begin{pgfscope}%
\pgfsetbuttcap%
\pgfsetroundjoin%
\definecolor{currentfill}{rgb}{0.000000,0.000000,0.000000}%
\pgfsetfillcolor{currentfill}%
\pgfsetlinewidth{0.602250pt}%
\definecolor{currentstroke}{rgb}{0.000000,0.000000,0.000000}%
\pgfsetstrokecolor{currentstroke}%
\pgfsetdash{}{0pt}%
\pgfsys@defobject{currentmarker}{\pgfqpoint{0.000000in}{-0.027778in}}{\pgfqpoint{0.000000in}{0.000000in}}{%
\pgfpathmoveto{\pgfqpoint{0.000000in}{0.000000in}}%
\pgfpathlineto{\pgfqpoint{0.000000in}{-0.027778in}}%
\pgfusepath{stroke,fill}%
}%
\begin{pgfscope}%
\pgfsys@transformshift{3.338197in}{0.565123in}%
\pgfsys@useobject{currentmarker}{}%
\end{pgfscope}%
\end{pgfscope}%
\begin{pgfscope}%
\pgfsetbuttcap%
\pgfsetroundjoin%
\definecolor{currentfill}{rgb}{0.000000,0.000000,0.000000}%
\pgfsetfillcolor{currentfill}%
\pgfsetlinewidth{0.602250pt}%
\definecolor{currentstroke}{rgb}{0.000000,0.000000,0.000000}%
\pgfsetstrokecolor{currentstroke}%
\pgfsetdash{}{0pt}%
\pgfsys@defobject{currentmarker}{\pgfqpoint{0.000000in}{-0.027778in}}{\pgfqpoint{0.000000in}{0.000000in}}{%
\pgfpathmoveto{\pgfqpoint{0.000000in}{0.000000in}}%
\pgfpathlineto{\pgfqpoint{0.000000in}{-0.027778in}}%
\pgfusepath{stroke,fill}%
}%
\begin{pgfscope}%
\pgfsys@transformshift{3.436452in}{0.565123in}%
\pgfsys@useobject{currentmarker}{}%
\end{pgfscope}%
\end{pgfscope}%
\begin{pgfscope}%
\pgfsetbuttcap%
\pgfsetroundjoin%
\definecolor{currentfill}{rgb}{0.000000,0.000000,0.000000}%
\pgfsetfillcolor{currentfill}%
\pgfsetlinewidth{0.602250pt}%
\definecolor{currentstroke}{rgb}{0.000000,0.000000,0.000000}%
\pgfsetstrokecolor{currentstroke}%
\pgfsetdash{}{0pt}%
\pgfsys@defobject{currentmarker}{\pgfqpoint{0.000000in}{-0.027778in}}{\pgfqpoint{0.000000in}{0.000000in}}{%
\pgfpathmoveto{\pgfqpoint{0.000000in}{0.000000in}}%
\pgfpathlineto{\pgfqpoint{0.000000in}{-0.027778in}}%
\pgfusepath{stroke,fill}%
}%
\begin{pgfscope}%
\pgfsys@transformshift{3.512664in}{0.565123in}%
\pgfsys@useobject{currentmarker}{}%
\end{pgfscope}%
\end{pgfscope}%
\begin{pgfscope}%
\pgfsetbuttcap%
\pgfsetroundjoin%
\definecolor{currentfill}{rgb}{0.000000,0.000000,0.000000}%
\pgfsetfillcolor{currentfill}%
\pgfsetlinewidth{0.602250pt}%
\definecolor{currentstroke}{rgb}{0.000000,0.000000,0.000000}%
\pgfsetstrokecolor{currentstroke}%
\pgfsetdash{}{0pt}%
\pgfsys@defobject{currentmarker}{\pgfqpoint{0.000000in}{-0.027778in}}{\pgfqpoint{0.000000in}{0.000000in}}{%
\pgfpathmoveto{\pgfqpoint{0.000000in}{0.000000in}}%
\pgfpathlineto{\pgfqpoint{0.000000in}{-0.027778in}}%
\pgfusepath{stroke,fill}%
}%
\begin{pgfscope}%
\pgfsys@transformshift{3.574934in}{0.565123in}%
\pgfsys@useobject{currentmarker}{}%
\end{pgfscope}%
\end{pgfscope}%
\begin{pgfscope}%
\pgfsetbuttcap%
\pgfsetroundjoin%
\definecolor{currentfill}{rgb}{0.000000,0.000000,0.000000}%
\pgfsetfillcolor{currentfill}%
\pgfsetlinewidth{0.602250pt}%
\definecolor{currentstroke}{rgb}{0.000000,0.000000,0.000000}%
\pgfsetstrokecolor{currentstroke}%
\pgfsetdash{}{0pt}%
\pgfsys@defobject{currentmarker}{\pgfqpoint{0.000000in}{-0.027778in}}{\pgfqpoint{0.000000in}{0.000000in}}{%
\pgfpathmoveto{\pgfqpoint{0.000000in}{0.000000in}}%
\pgfpathlineto{\pgfqpoint{0.000000in}{-0.027778in}}%
\pgfusepath{stroke,fill}%
}%
\begin{pgfscope}%
\pgfsys@transformshift{3.627583in}{0.565123in}%
\pgfsys@useobject{currentmarker}{}%
\end{pgfscope}%
\end{pgfscope}%
\begin{pgfscope}%
\pgfsetbuttcap%
\pgfsetroundjoin%
\definecolor{currentfill}{rgb}{0.000000,0.000000,0.000000}%
\pgfsetfillcolor{currentfill}%
\pgfsetlinewidth{0.602250pt}%
\definecolor{currentstroke}{rgb}{0.000000,0.000000,0.000000}%
\pgfsetstrokecolor{currentstroke}%
\pgfsetdash{}{0pt}%
\pgfsys@defobject{currentmarker}{\pgfqpoint{0.000000in}{-0.027778in}}{\pgfqpoint{0.000000in}{0.000000in}}{%
\pgfpathmoveto{\pgfqpoint{0.000000in}{0.000000in}}%
\pgfpathlineto{\pgfqpoint{0.000000in}{-0.027778in}}%
\pgfusepath{stroke,fill}%
}%
\begin{pgfscope}%
\pgfsys@transformshift{3.673189in}{0.565123in}%
\pgfsys@useobject{currentmarker}{}%
\end{pgfscope}%
\end{pgfscope}%
\begin{pgfscope}%
\pgfsetbuttcap%
\pgfsetroundjoin%
\definecolor{currentfill}{rgb}{0.000000,0.000000,0.000000}%
\pgfsetfillcolor{currentfill}%
\pgfsetlinewidth{0.602250pt}%
\definecolor{currentstroke}{rgb}{0.000000,0.000000,0.000000}%
\pgfsetstrokecolor{currentstroke}%
\pgfsetdash{}{0pt}%
\pgfsys@defobject{currentmarker}{\pgfqpoint{0.000000in}{-0.027778in}}{\pgfqpoint{0.000000in}{0.000000in}}{%
\pgfpathmoveto{\pgfqpoint{0.000000in}{0.000000in}}%
\pgfpathlineto{\pgfqpoint{0.000000in}{-0.027778in}}%
\pgfusepath{stroke,fill}%
}%
\begin{pgfscope}%
\pgfsys@transformshift{3.713417in}{0.565123in}%
\pgfsys@useobject{currentmarker}{}%
\end{pgfscope}%
\end{pgfscope}%
\begin{pgfscope}%
\definecolor{textcolor}{rgb}{0.000000,0.000000,0.000000}%
\pgfsetstrokecolor{textcolor}%
\pgfsetfillcolor{textcolor}%
\pgftext[x=2.176553in,y=0.288889in,,top]{\color{textcolor}\rmfamily\fontsize{10.000000}{12.000000}\selectfont Breakpoint (N)}%
\end{pgfscope}%
\begin{pgfscope}%
\pgfsetbuttcap%
\pgfsetroundjoin%
\definecolor{currentfill}{rgb}{0.000000,0.000000,0.000000}%
\pgfsetfillcolor{currentfill}%
\pgfsetlinewidth{0.803000pt}%
\definecolor{currentstroke}{rgb}{0.000000,0.000000,0.000000}%
\pgfsetstrokecolor{currentstroke}%
\pgfsetdash{}{0pt}%
\pgfsys@defobject{currentmarker}{\pgfqpoint{-0.048611in}{0.000000in}}{\pgfqpoint{0.000000in}{0.000000in}}{%
\pgfpathmoveto{\pgfqpoint{0.000000in}{0.000000in}}%
\pgfpathlineto{\pgfqpoint{-0.048611in}{0.000000in}}%
\pgfusepath{stroke,fill}%
}%
\begin{pgfscope}%
\pgfsys@transformshift{0.603704in}{0.565123in}%
\pgfsys@useobject{currentmarker}{}%
\end{pgfscope}%
\end{pgfscope}%
\begin{pgfscope}%
\definecolor{textcolor}{rgb}{0.000000,0.000000,0.000000}%
\pgfsetstrokecolor{textcolor}%
\pgfsetfillcolor{textcolor}%
\pgftext[x=0.329012in, y=0.516898in, left, base]{\color{textcolor}\rmfamily\fontsize{10.000000}{12.000000}\selectfont \(\displaystyle {0.0}\)}%
\end{pgfscope}%
\begin{pgfscope}%
\pgfsetbuttcap%
\pgfsetroundjoin%
\definecolor{currentfill}{rgb}{0.000000,0.000000,0.000000}%
\pgfsetfillcolor{currentfill}%
\pgfsetlinewidth{0.803000pt}%
\definecolor{currentstroke}{rgb}{0.000000,0.000000,0.000000}%
\pgfsetstrokecolor{currentstroke}%
\pgfsetdash{}{0pt}%
\pgfsys@defobject{currentmarker}{\pgfqpoint{-0.048611in}{0.000000in}}{\pgfqpoint{0.000000in}{0.000000in}}{%
\pgfpathmoveto{\pgfqpoint{0.000000in}{0.000000in}}%
\pgfpathlineto{\pgfqpoint{-0.048611in}{0.000000in}}%
\pgfusepath{stroke,fill}%
}%
\begin{pgfscope}%
\pgfsys@transformshift{0.603704in}{1.012564in}%
\pgfsys@useobject{currentmarker}{}%
\end{pgfscope}%
\end{pgfscope}%
\begin{pgfscope}%
\definecolor{textcolor}{rgb}{0.000000,0.000000,0.000000}%
\pgfsetstrokecolor{textcolor}%
\pgfsetfillcolor{textcolor}%
\pgftext[x=0.329012in, y=0.964338in, left, base]{\color{textcolor}\rmfamily\fontsize{10.000000}{12.000000}\selectfont \(\displaystyle {0.2}\)}%
\end{pgfscope}%
\begin{pgfscope}%
\pgfsetbuttcap%
\pgfsetroundjoin%
\definecolor{currentfill}{rgb}{0.000000,0.000000,0.000000}%
\pgfsetfillcolor{currentfill}%
\pgfsetlinewidth{0.803000pt}%
\definecolor{currentstroke}{rgb}{0.000000,0.000000,0.000000}%
\pgfsetstrokecolor{currentstroke}%
\pgfsetdash{}{0pt}%
\pgfsys@defobject{currentmarker}{\pgfqpoint{-0.048611in}{0.000000in}}{\pgfqpoint{0.000000in}{0.000000in}}{%
\pgfpathmoveto{\pgfqpoint{0.000000in}{0.000000in}}%
\pgfpathlineto{\pgfqpoint{-0.048611in}{0.000000in}}%
\pgfusepath{stroke,fill}%
}%
\begin{pgfscope}%
\pgfsys@transformshift{0.603704in}{1.460004in}%
\pgfsys@useobject{currentmarker}{}%
\end{pgfscope}%
\end{pgfscope}%
\begin{pgfscope}%
\definecolor{textcolor}{rgb}{0.000000,0.000000,0.000000}%
\pgfsetstrokecolor{textcolor}%
\pgfsetfillcolor{textcolor}%
\pgftext[x=0.329012in, y=1.411779in, left, base]{\color{textcolor}\rmfamily\fontsize{10.000000}{12.000000}\selectfont \(\displaystyle {0.4}\)}%
\end{pgfscope}%
\begin{pgfscope}%
\pgfsetbuttcap%
\pgfsetroundjoin%
\definecolor{currentfill}{rgb}{0.000000,0.000000,0.000000}%
\pgfsetfillcolor{currentfill}%
\pgfsetlinewidth{0.803000pt}%
\definecolor{currentstroke}{rgb}{0.000000,0.000000,0.000000}%
\pgfsetstrokecolor{currentstroke}%
\pgfsetdash{}{0pt}%
\pgfsys@defobject{currentmarker}{\pgfqpoint{-0.048611in}{0.000000in}}{\pgfqpoint{0.000000in}{0.000000in}}{%
\pgfpathmoveto{\pgfqpoint{0.000000in}{0.000000in}}%
\pgfpathlineto{\pgfqpoint{-0.048611in}{0.000000in}}%
\pgfusepath{stroke,fill}%
}%
\begin{pgfscope}%
\pgfsys@transformshift{0.603704in}{1.907444in}%
\pgfsys@useobject{currentmarker}{}%
\end{pgfscope}%
\end{pgfscope}%
\begin{pgfscope}%
\definecolor{textcolor}{rgb}{0.000000,0.000000,0.000000}%
\pgfsetstrokecolor{textcolor}%
\pgfsetfillcolor{textcolor}%
\pgftext[x=0.329012in, y=1.859219in, left, base]{\color{textcolor}\rmfamily\fontsize{10.000000}{12.000000}\selectfont \(\displaystyle {0.6}\)}%
\end{pgfscope}%
\begin{pgfscope}%
\pgfsetbuttcap%
\pgfsetroundjoin%
\definecolor{currentfill}{rgb}{0.000000,0.000000,0.000000}%
\pgfsetfillcolor{currentfill}%
\pgfsetlinewidth{0.803000pt}%
\definecolor{currentstroke}{rgb}{0.000000,0.000000,0.000000}%
\pgfsetstrokecolor{currentstroke}%
\pgfsetdash{}{0pt}%
\pgfsys@defobject{currentmarker}{\pgfqpoint{-0.048611in}{0.000000in}}{\pgfqpoint{0.000000in}{0.000000in}}{%
\pgfpathmoveto{\pgfqpoint{0.000000in}{0.000000in}}%
\pgfpathlineto{\pgfqpoint{-0.048611in}{0.000000in}}%
\pgfusepath{stroke,fill}%
}%
\begin{pgfscope}%
\pgfsys@transformshift{0.603704in}{2.354885in}%
\pgfsys@useobject{currentmarker}{}%
\end{pgfscope}%
\end{pgfscope}%
\begin{pgfscope}%
\definecolor{textcolor}{rgb}{0.000000,0.000000,0.000000}%
\pgfsetstrokecolor{textcolor}%
\pgfsetfillcolor{textcolor}%
\pgftext[x=0.329012in, y=2.306659in, left, base]{\color{textcolor}\rmfamily\fontsize{10.000000}{12.000000}\selectfont \(\displaystyle {0.8}\)}%
\end{pgfscope}%
\begin{pgfscope}%
\pgfsetbuttcap%
\pgfsetroundjoin%
\definecolor{currentfill}{rgb}{0.000000,0.000000,0.000000}%
\pgfsetfillcolor{currentfill}%
\pgfsetlinewidth{0.803000pt}%
\definecolor{currentstroke}{rgb}{0.000000,0.000000,0.000000}%
\pgfsetstrokecolor{currentstroke}%
\pgfsetdash{}{0pt}%
\pgfsys@defobject{currentmarker}{\pgfqpoint{-0.048611in}{0.000000in}}{\pgfqpoint{0.000000in}{0.000000in}}{%
\pgfpathmoveto{\pgfqpoint{0.000000in}{0.000000in}}%
\pgfpathlineto{\pgfqpoint{-0.048611in}{0.000000in}}%
\pgfusepath{stroke,fill}%
}%
\begin{pgfscope}%
\pgfsys@transformshift{0.603704in}{2.802325in}%
\pgfsys@useobject{currentmarker}{}%
\end{pgfscope}%
\end{pgfscope}%
\begin{pgfscope}%
\definecolor{textcolor}{rgb}{0.000000,0.000000,0.000000}%
\pgfsetstrokecolor{textcolor}%
\pgfsetfillcolor{textcolor}%
\pgftext[x=0.329012in, y=2.754100in, left, base]{\color{textcolor}\rmfamily\fontsize{10.000000}{12.000000}\selectfont \(\displaystyle {1.0}\)}%
\end{pgfscope}%
\begin{pgfscope}%
\definecolor{textcolor}{rgb}{0.000000,0.000000,0.000000}%
\pgfsetstrokecolor{textcolor}%
\pgfsetfillcolor{textcolor}%
\pgftext[x=0.273457in,y=1.706096in,,bottom,rotate=90.000000]{\color{textcolor}\rmfamily\fontsize{10.000000}{12.000000}\selectfont Proportion of commodities}%
\end{pgfscope}%
\begin{pgfscope}%
\pgfpathrectangle{\pgfqpoint{0.603704in}{0.565123in}}{\pgfqpoint{3.145698in}{2.281946in}}%
\pgfusepath{clip}%
\pgfsetbuttcap%
\pgfsetroundjoin%
\pgfsetlinewidth{1.505625pt}%
\definecolor{currentstroke}{rgb}{0.121569,0.466667,0.705882}%
\pgfsetstrokecolor{currentstroke}%
\pgfsetdash{{1.500000pt}{2.475000pt}}{0.000000pt}%
\pgfpathmoveto{\pgfqpoint{0.603704in}{0.888399in}}%
\pgfpathlineto{\pgfqpoint{0.840441in}{1.104289in}}%
\pgfpathlineto{\pgfqpoint{0.978924in}{1.250826in}}%
\pgfpathlineto{\pgfqpoint{1.077179in}{1.330246in}}%
\pgfpathlineto{\pgfqpoint{1.153391in}{1.411904in}}%
\pgfpathlineto{\pgfqpoint{1.215661in}{1.490206in}}%
\pgfpathlineto{\pgfqpoint{1.268310in}{1.542780in}}%
\pgfpathlineto{\pgfqpoint{1.313916in}{1.599829in}}%
\pgfpathlineto{\pgfqpoint{1.354144in}{1.630031in}}%
\pgfpathlineto{\pgfqpoint{1.390129in}{1.660233in}}%
\pgfpathlineto{\pgfqpoint{1.452399in}{1.756433in}}%
\pgfpathlineto{\pgfqpoint{1.528611in}{1.851514in}}%
\pgfpathlineto{\pgfqpoint{1.626866in}{1.946595in}}%
\pgfpathlineto{\pgfqpoint{1.703078in}{2.024897in}}%
\pgfpathlineto{\pgfqpoint{1.765348in}{2.088658in}}%
\pgfpathlineto{\pgfqpoint{1.863603in}{2.179264in}}%
\pgfpathlineto{\pgfqpoint{1.939816in}{2.253092in}}%
\pgfpathlineto{\pgfqpoint{2.002086in}{2.309022in}}%
\pgfpathlineto{\pgfqpoint{2.100341in}{2.388443in}}%
\pgfpathlineto{\pgfqpoint{2.176553in}{2.445491in}}%
\pgfpathlineto{\pgfqpoint{2.238823in}{2.502540in}}%
\pgfpathlineto{\pgfqpoint{2.315035in}{2.566300in}}%
\pgfpathlineto{\pgfqpoint{2.413290in}{2.630060in}}%
\pgfpathlineto{\pgfqpoint{2.489503in}{2.669211in}}%
\pgfpathlineto{\pgfqpoint{2.551773in}{2.699414in}}%
\pgfpathlineto{\pgfqpoint{2.650028in}{2.741920in}}%
\pgfpathlineto{\pgfqpoint{2.726240in}{2.759818in}}%
\pgfpathlineto{\pgfqpoint{2.788510in}{2.775478in}}%
\pgfpathlineto{\pgfqpoint{2.886765in}{2.794495in}}%
\pgfpathlineto{\pgfqpoint{2.962977in}{2.801206in}}%
\pgfpathlineto{\pgfqpoint{3.025247in}{2.801206in}}%
\pgfpathlineto{\pgfqpoint{3.101460in}{2.801206in}}%
\pgfpathlineto{\pgfqpoint{3.199715in}{2.802325in}}%
\pgfpathlineto{\pgfqpoint{3.275927in}{2.802325in}}%
\pgfpathlineto{\pgfqpoint{3.338197in}{2.802325in}}%
\pgfpathlineto{\pgfqpoint{3.436452in}{2.802325in}}%
\pgfpathlineto{\pgfqpoint{3.512664in}{2.802325in}}%
\pgfpathlineto{\pgfqpoint{3.574934in}{2.802325in}}%
\pgfpathlineto{\pgfqpoint{3.673189in}{2.802325in}}%
\pgfpathlineto{\pgfqpoint{3.749402in}{2.802325in}}%
\pgfusepath{stroke}%
\end{pgfscope}%
\begin{pgfscope}%
\pgfpathrectangle{\pgfqpoint{0.603704in}{0.565123in}}{\pgfqpoint{3.145698in}{2.281946in}}%
\pgfusepath{clip}%
\pgfsetrectcap%
\pgfsetroundjoin%
\pgfsetlinewidth{1.505625pt}%
\definecolor{currentstroke}{rgb}{1.000000,0.498039,0.054902}%
\pgfsetstrokecolor{currentstroke}%
\pgfsetdash{}{0pt}%
\pgfpathmoveto{\pgfqpoint{0.603704in}{0.696000in}}%
\pgfpathlineto{\pgfqpoint{0.840441in}{0.808978in}}%
\pgfpathlineto{\pgfqpoint{0.978924in}{0.916364in}}%
\pgfpathlineto{\pgfqpoint{1.077179in}{0.965582in}}%
\pgfpathlineto{\pgfqpoint{1.153391in}{1.017038in}}%
\pgfpathlineto{\pgfqpoint{1.215661in}{1.066256in}}%
\pgfpathlineto{\pgfqpoint{1.268310in}{1.118831in}}%
\pgfpathlineto{\pgfqpoint{1.313916in}{1.146796in}}%
\pgfpathlineto{\pgfqpoint{1.354144in}{1.187065in}}%
\pgfpathlineto{\pgfqpoint{1.390129in}{1.218386in}}%
\pgfpathlineto{\pgfqpoint{1.452399in}{1.266486in}}%
\pgfpathlineto{\pgfqpoint{1.528611in}{1.326890in}}%
\pgfpathlineto{\pgfqpoint{1.626866in}{1.402955in}}%
\pgfpathlineto{\pgfqpoint{1.703078in}{1.480139in}}%
\pgfpathlineto{\pgfqpoint{1.765348in}{1.539425in}}%
\pgfpathlineto{\pgfqpoint{1.863603in}{1.613252in}}%
\pgfpathlineto{\pgfqpoint{1.939816in}{1.675894in}}%
\pgfpathlineto{\pgfqpoint{2.002086in}{1.731824in}}%
\pgfpathlineto{\pgfqpoint{2.100341in}{1.805652in}}%
\pgfpathlineto{\pgfqpoint{2.176553in}{1.891784in}}%
\pgfpathlineto{\pgfqpoint{2.238823in}{1.956663in}}%
\pgfpathlineto{\pgfqpoint{2.315035in}{2.032728in}}%
\pgfpathlineto{\pgfqpoint{2.413290in}{2.122216in}}%
\pgfpathlineto{\pgfqpoint{2.489503in}{2.194925in}}%
\pgfpathlineto{\pgfqpoint{2.551773in}{2.246380in}}%
\pgfpathlineto{\pgfqpoint{2.650028in}{2.313496in}}%
\pgfpathlineto{\pgfqpoint{2.726240in}{2.373901in}}%
\pgfpathlineto{\pgfqpoint{2.788510in}{2.414170in}}%
\pgfpathlineto{\pgfqpoint{2.886765in}{2.496947in}}%
\pgfpathlineto{\pgfqpoint{2.962977in}{2.532742in}}%
\pgfpathlineto{\pgfqpoint{3.025247in}{2.571893in}}%
\pgfpathlineto{\pgfqpoint{3.101460in}{2.613281in}}%
\pgfpathlineto{\pgfqpoint{3.199715in}{2.661381in}}%
\pgfpathlineto{\pgfqpoint{3.275927in}{2.694939in}}%
\pgfpathlineto{\pgfqpoint{3.338197in}{2.722904in}}%
\pgfpathlineto{\pgfqpoint{3.436452in}{2.739683in}}%
\pgfpathlineto{\pgfqpoint{3.512664in}{2.760937in}}%
\pgfpathlineto{\pgfqpoint{3.574934in}{2.771004in}}%
\pgfpathlineto{\pgfqpoint{3.673189in}{2.782190in}}%
\pgfpathlineto{\pgfqpoint{3.749402in}{2.792257in}}%
\pgfusepath{stroke}%
\end{pgfscope}%
\begin{pgfscope}%
\pgfpathrectangle{\pgfqpoint{0.603704in}{0.565123in}}{\pgfqpoint{3.145698in}{2.281946in}}%
\pgfusepath{clip}%
\pgfsetbuttcap%
\pgfsetroundjoin%
\pgfsetlinewidth{1.505625pt}%
\definecolor{currentstroke}{rgb}{0.172549,0.627451,0.172549}%
\pgfsetstrokecolor{currentstroke}%
\pgfsetdash{{9.600000pt}{2.400000pt}{1.500000pt}{2.400000pt}}{0.000000pt}%
\pgfpathmoveto{\pgfqpoint{0.603704in}{0.647900in}}%
\pgfpathlineto{\pgfqpoint{0.840441in}{0.725083in}}%
\pgfpathlineto{\pgfqpoint{0.978924in}{0.787725in}}%
\pgfpathlineto{\pgfqpoint{1.077179in}{0.849248in}}%
\pgfpathlineto{\pgfqpoint{1.153391in}{0.895111in}}%
\pgfpathlineto{\pgfqpoint{1.215661in}{0.928669in}}%
\pgfpathlineto{\pgfqpoint{1.268310in}{0.956634in}}%
\pgfpathlineto{\pgfqpoint{1.313916in}{0.990192in}}%
\pgfpathlineto{\pgfqpoint{1.354144in}{1.015919in}}%
\pgfpathlineto{\pgfqpoint{1.390129in}{1.040529in}}%
\pgfpathlineto{\pgfqpoint{1.452399in}{1.094221in}}%
\pgfpathlineto{\pgfqpoint{1.528611in}{1.151270in}}%
\pgfpathlineto{\pgfqpoint{1.626866in}{1.223979in}}%
\pgfpathlineto{\pgfqpoint{1.703078in}{1.278791in}}%
\pgfpathlineto{\pgfqpoint{1.765348in}{1.324653in}}%
\pgfpathlineto{\pgfqpoint{1.863603in}{1.415260in}}%
\pgfpathlineto{\pgfqpoint{1.939816in}{1.495799in}}%
\pgfpathlineto{\pgfqpoint{2.002086in}{1.546136in}}%
\pgfpathlineto{\pgfqpoint{2.100341in}{1.619964in}}%
\pgfpathlineto{\pgfqpoint{2.176553in}{1.692673in}}%
\pgfpathlineto{\pgfqpoint{2.238823in}{1.748603in}}%
\pgfpathlineto{\pgfqpoint{2.315035in}{1.810126in}}%
\pgfpathlineto{\pgfqpoint{2.413290in}{1.886191in}}%
\pgfpathlineto{\pgfqpoint{2.489503in}{1.947714in}}%
\pgfpathlineto{\pgfqpoint{2.551773in}{1.996932in}}%
\pgfpathlineto{\pgfqpoint{2.650028in}{2.061811in}}%
\pgfpathlineto{\pgfqpoint{2.726240in}{2.125571in}}%
\pgfpathlineto{\pgfqpoint{2.788510in}{2.174790in}}%
\pgfpathlineto{\pgfqpoint{2.886765in}{2.243024in}}%
\pgfpathlineto{\pgfqpoint{2.962977in}{2.294480in}}%
\pgfpathlineto{\pgfqpoint{3.025247in}{2.343699in}}%
\pgfpathlineto{\pgfqpoint{3.101460in}{2.395154in}}%
\pgfpathlineto{\pgfqpoint{3.199715in}{2.457796in}}%
\pgfpathlineto{\pgfqpoint{3.275927in}{2.513726in}}%
\pgfpathlineto{\pgfqpoint{3.338197in}{2.549521in}}%
\pgfpathlineto{\pgfqpoint{3.436452in}{2.602095in}}%
\pgfpathlineto{\pgfqpoint{3.512664in}{2.647958in}}%
\pgfpathlineto{\pgfqpoint{3.574934in}{2.684872in}}%
\pgfpathlineto{\pgfqpoint{3.673189in}{2.727379in}}%
\pgfpathlineto{\pgfqpoint{3.749402in}{2.747513in}}%
\pgfusepath{stroke}%
\end{pgfscope}%
\begin{pgfscope}%
\pgfpathrectangle{\pgfqpoint{0.603704in}{0.565123in}}{\pgfqpoint{3.145698in}{2.281946in}}%
\pgfusepath{clip}%
\pgfsetbuttcap%
\pgfsetroundjoin%
\pgfsetlinewidth{1.505625pt}%
\definecolor{currentstroke}{rgb}{0.839216,0.152941,0.156863}%
\pgfsetstrokecolor{currentstroke}%
\pgfsetdash{{5.550000pt}{2.400000pt}}{0.000000pt}%
\pgfpathmoveto{\pgfqpoint{0.603704in}{0.757523in}}%
\pgfpathlineto{\pgfqpoint{0.840441in}{0.900704in}}%
\pgfpathlineto{\pgfqpoint{0.978924in}{0.992429in}}%
\pgfpathlineto{\pgfqpoint{1.077179in}{1.066256in}}%
\pgfpathlineto{\pgfqpoint{1.153391in}{1.116593in}}%
\pgfpathlineto{\pgfqpoint{1.215661in}{1.164693in}}%
\pgfpathlineto{\pgfqpoint{1.268310in}{1.193777in}}%
\pgfpathlineto{\pgfqpoint{1.313916in}{1.227335in}}%
\pgfpathlineto{\pgfqpoint{1.354144in}{1.256419in}}%
\pgfpathlineto{\pgfqpoint{1.390129in}{1.272079in}}%
\pgfpathlineto{\pgfqpoint{1.452399in}{1.320179in}}%
\pgfpathlineto{\pgfqpoint{1.528611in}{1.376109in}}%
\pgfpathlineto{\pgfqpoint{1.626866in}{1.457767in}}%
\pgfpathlineto{\pgfqpoint{1.703078in}{1.512578in}}%
\pgfpathlineto{\pgfqpoint{1.765348in}{1.543899in}}%
\pgfpathlineto{\pgfqpoint{1.863603in}{1.613252in}}%
\pgfpathlineto{\pgfqpoint{1.939816in}{1.671419in}}%
\pgfpathlineto{\pgfqpoint{2.002086in}{1.715045in}}%
\pgfpathlineto{\pgfqpoint{2.100341in}{1.791110in}}%
\pgfpathlineto{\pgfqpoint{2.176553in}{1.847040in}}%
\pgfpathlineto{\pgfqpoint{2.238823in}{1.904088in}}%
\pgfpathlineto{\pgfqpoint{2.315035in}{1.954425in}}%
\pgfpathlineto{\pgfqpoint{2.413290in}{2.055100in}}%
\pgfpathlineto{\pgfqpoint{2.489503in}{2.124453in}}%
\pgfpathlineto{\pgfqpoint{2.551773in}{2.178146in}}%
\pgfpathlineto{\pgfqpoint{2.650028in}{2.291124in}}%
\pgfpathlineto{\pgfqpoint{2.726240in}{2.367189in}}%
\pgfpathlineto{\pgfqpoint{2.788510in}{2.434305in}}%
\pgfpathlineto{\pgfqpoint{2.886765in}{2.527149in}}%
\pgfpathlineto{\pgfqpoint{2.962977in}{2.589791in}}%
\pgfpathlineto{\pgfqpoint{3.025247in}{2.630060in}}%
\pgfpathlineto{\pgfqpoint{3.101460in}{2.674804in}}%
\pgfpathlineto{\pgfqpoint{3.199715in}{2.721786in}}%
\pgfpathlineto{\pgfqpoint{3.275927in}{2.745276in}}%
\pgfpathlineto{\pgfqpoint{3.338197in}{2.755344in}}%
\pgfpathlineto{\pgfqpoint{3.436452in}{2.774360in}}%
\pgfpathlineto{\pgfqpoint{3.512664in}{2.782190in}}%
\pgfpathlineto{\pgfqpoint{3.574934in}{2.786664in}}%
\pgfpathlineto{\pgfqpoint{3.673189in}{2.794495in}}%
\pgfpathlineto{\pgfqpoint{3.749402in}{2.797850in}}%
\pgfusepath{stroke}%
\end{pgfscope}%
\begin{pgfscope}%
\pgfsetrectcap%
\pgfsetmiterjoin%
\pgfsetlinewidth{0.803000pt}%
\definecolor{currentstroke}{rgb}{0.000000,0.000000,0.000000}%
\pgfsetstrokecolor{currentstroke}%
\pgfsetdash{}{0pt}%
\pgfpathmoveto{\pgfqpoint{0.603704in}{0.565123in}}%
\pgfpathlineto{\pgfqpoint{0.603704in}{2.847069in}}%
\pgfusepath{stroke}%
\end{pgfscope}%
\begin{pgfscope}%
\pgfsetrectcap%
\pgfsetmiterjoin%
\pgfsetlinewidth{0.803000pt}%
\definecolor{currentstroke}{rgb}{0.000000,0.000000,0.000000}%
\pgfsetstrokecolor{currentstroke}%
\pgfsetdash{}{0pt}%
\pgfpathmoveto{\pgfqpoint{3.749402in}{0.565123in}}%
\pgfpathlineto{\pgfqpoint{3.749402in}{2.847069in}}%
\pgfusepath{stroke}%
\end{pgfscope}%
\begin{pgfscope}%
\pgfsetrectcap%
\pgfsetmiterjoin%
\pgfsetlinewidth{0.803000pt}%
\definecolor{currentstroke}{rgb}{0.000000,0.000000,0.000000}%
\pgfsetstrokecolor{currentstroke}%
\pgfsetdash{}{0pt}%
\pgfpathmoveto{\pgfqpoint{0.603704in}{0.565123in}}%
\pgfpathlineto{\pgfqpoint{3.749402in}{0.565123in}}%
\pgfusepath{stroke}%
\end{pgfscope}%
\begin{pgfscope}%
\pgfsetrectcap%
\pgfsetmiterjoin%
\pgfsetlinewidth{0.803000pt}%
\definecolor{currentstroke}{rgb}{0.000000,0.000000,0.000000}%
\pgfsetstrokecolor{currentstroke}%
\pgfsetdash{}{0pt}%
\pgfpathmoveto{\pgfqpoint{0.603704in}{2.847069in}}%
\pgfpathlineto{\pgfqpoint{3.749402in}{2.847069in}}%
\pgfusepath{stroke}%
\end{pgfscope}%
\begin{pgfscope}%
\pgfsetbuttcap%
\pgfsetroundjoin%
\pgfsetlinewidth{1.505625pt}%
\definecolor{currentstroke}{rgb}{0.121569,0.466667,0.705882}%
\pgfsetstrokecolor{currentstroke}%
\pgfsetdash{{1.500000pt}{2.475000pt}}{0.000000pt}%
\pgfpathmoveto{\pgfqpoint{3.245677in}{1.190432in}}%
\pgfpathlineto{\pgfqpoint{3.467900in}{1.190432in}}%
\pgfusepath{stroke}%
\end{pgfscope}%
\begin{pgfscope}%
\definecolor{textcolor}{rgb}{0.000000,0.000000,0.000000}%
\pgfsetstrokecolor{textcolor}%
\pgfsetfillcolor{textcolor}%
\pgftext[x=3.556788in,y=1.151543in,left,base]{\color{textcolor}\rmfamily\fontsize{8.000000}{9.600000}\selectfont G}%
\end{pgfscope}%
\begin{pgfscope}%
\pgfsetrectcap%
\pgfsetroundjoin%
\pgfsetlinewidth{1.505625pt}%
\definecolor{currentstroke}{rgb}{1.000000,0.498039,0.054902}%
\pgfsetstrokecolor{currentstroke}%
\pgfsetdash{}{0pt}%
\pgfpathmoveto{\pgfqpoint{3.245677in}{1.035494in}}%
\pgfpathlineto{\pgfqpoint{3.467900in}{1.035494in}}%
\pgfusepath{stroke}%
\end{pgfscope}%
\begin{pgfscope}%
\definecolor{textcolor}{rgb}{0.000000,0.000000,0.000000}%
\pgfsetstrokecolor{textcolor}%
\pgfsetfillcolor{textcolor}%
\pgftext[x=3.556788in,y=0.996605in,left,base]{\color{textcolor}\rmfamily\fontsize{8.000000}{9.600000}\selectfont H}%
\end{pgfscope}%
\begin{pgfscope}%
\pgfsetbuttcap%
\pgfsetroundjoin%
\pgfsetlinewidth{1.505625pt}%
\definecolor{currentstroke}{rgb}{0.172549,0.627451,0.172549}%
\pgfsetstrokecolor{currentstroke}%
\pgfsetdash{{9.600000pt}{2.400000pt}{1.500000pt}{2.400000pt}}{0.000000pt}%
\pgfpathmoveto{\pgfqpoint{3.245677in}{0.880556in}}%
\pgfpathlineto{\pgfqpoint{3.467900in}{0.880556in}}%
\pgfusepath{stroke}%
\end{pgfscope}%
\begin{pgfscope}%
\definecolor{textcolor}{rgb}{0.000000,0.000000,0.000000}%
\pgfsetstrokecolor{textcolor}%
\pgfsetfillcolor{textcolor}%
\pgftext[x=3.556788in,y=0.841667in,left,base]{\color{textcolor}\rmfamily\fontsize{8.000000}{9.600000}\selectfont D}%
\end{pgfscope}%
\begin{pgfscope}%
\pgfsetbuttcap%
\pgfsetroundjoin%
\pgfsetlinewidth{1.505625pt}%
\definecolor{currentstroke}{rgb}{0.839216,0.152941,0.156863}%
\pgfsetstrokecolor{currentstroke}%
\pgfsetdash{{5.550000pt}{2.400000pt}}{0.000000pt}%
\pgfpathmoveto{\pgfqpoint{3.245677in}{0.725617in}}%
\pgfpathlineto{\pgfqpoint{3.467900in}{0.725617in}}%
\pgfusepath{stroke}%
\end{pgfscope}%
\begin{pgfscope}%
\definecolor{textcolor}{rgb}{0.000000,0.000000,0.000000}%
\pgfsetstrokecolor{textcolor}%
\pgfsetfillcolor{textcolor}%
\pgftext[x=3.556788in,y=0.686728in,left,base]{\color{textcolor}\rmfamily\fontsize{8.000000}{9.600000}\selectfont V}%
\end{pgfscope}%
\end{pgfpicture}%
\makeatother%
\endgroup%